\documentclass[12pt]{article}
\usepackage{amssymb,amsmath,amsthm,epsfig}
\usepackage{latexsym, enumerate}
\usepackage{eepic}
\usepackage{epic}
\usepackage{graphicx}
\usepackage{color}
\usepackage{ifpdf}
\usepackage{subfigure}
\usepackage{tikz}
\usepackage{dsfont}
\usepackage{multirow}
\usepackage{makecell}
\usepackage{algorithm}
\usepackage{bm}
\usepackage{multirow}
\usepackage{color}
\usepackage[colorlinks, linkcolor=blue,anchorcolor=blue,citecolor=blue,urlcolor=blue]{hyperref}
\usepackage{appendix}
\usepackage{comment}
\usepackage{tikz}
\usetikzlibrary{decorations.pathmorphing}
\graphicspath{{figs/}}

\theoremstyle{plain}
\newtheorem{lem}{Lemma}[section]
\newtheorem{thm}[lem]{Theorem}

\theoremstyle{definition}
\newtheorem{defn}{Definition}[section]

\theoremstyle{remark}
\newtheorem{rem}{Remark}[section]

\newcommand{\p}{\partial}
\newcommand{\ds}{\displaystyle}

\newcommand{\ms}{\medskip}
\newcommand{\R}{ \mathbb{R}}

\def \e{\ensuremath{\mathrm{e}}}
\def \i{\ensuremath{\mathrm{i}}}
\def \d{\ensuremath{\mathrm{d}}}
\begin{document}

\title{ \large\bf Quantitative symmetry-breaking and nonlinear harmonic generation in plasmonics}
\author{
Hongyu Liu\thanks{Department of Mathematics, City University of Hong Kong, Kowloon, Hong Kong, China.\ \ Email: hongyu.liuip@gmail.com; hongyliu@cityu.edu.hk}
\and
Zhi-Qiang Miao\thanks
{School of Mathematics and Statistics, Central South University, Changsha 410083, China, and
Department of Mathematics, City University of Hong Kong, Kowloon, Hong Kong, China.
Email: zhiqmiao@csu.edu.cn; zhiqmiao@cityu.edu.hk}
\and
Jingfeng Yao\thanks{School of Physics, Harbin Institute of Technology, Harbin, China.\ \ Email: yaojf@hit.edu.cn}
\and
Chengxun Yuan\thanks{School of Physics, Harbin Institute of Technology, Harbin, China. \ \ Email: yuancx@hit.edu.cn}
\and 
Guang-Hui Zheng\thanks{School of Mathematics, Hunan University, Changsha 410082, China. Email: zhenggh2012@hnu.edu.cn; zhgh1980@163.com}
}
\date{}%
\maketitle
% ----------------------------------------------------------------
\begin{abstract}
We develop a quantitative mathematical theory that offers new perspectives on nonlinear harmonic generation in plasmonic structures arising from symmetry breaking. Focusing on second harmonic generation--the most fundamental process and the most extensively studied owing to its practical significance--we establish a theoretical framework that can be readily extended to higher-order harmonics. We investigate the plasmonic system in the static regime using a columnar nanowire with \(n\)-fold rotational symmetry (\(n \in \mathbb{N}\)) and construct a phenomenological model in which the second harmonic response originates from nonlinear sources confined to a selvedge region near the surface. By introducing a notion of symmetry degree grounded in group theory, we precisely quantify the second harmonic generation in terms of multipolar contributions. Our theory complements existing physical descriptions of this practically important phenomenon and provides a rigorous account of how nonlinear optical efficiency depends on shape, size, symmetry, and defects in plasmonic structures.
\end{abstract}

\tableofcontents

\section{Introduction}
Nonlinear optics is a fundamental branch of modern photonics, with its origins dating back to the classic 1961 experiment by Franken \emph{et al.} \cite{Franken1961}, in which optical second harmonic generation was first observed by directing a ruby laser \cite{Maiman1960} into a quartz crystal. Unlike classical optics, which deals with linear interactions where light passes through a medium without changing its frequency, nonlinear optics \cite{Moloney2004, Boyd2008} explores phenomena where intense light fundamentally alters the optical properties of the material, causing its response to an incident field to depend nonlinearly on the field's strength. This field has not only greatly expanded the tunable range of laser frequencies but also provided essential technological support for optical communications, quantum optics, and precision measurement. However, nonlinear optical effects arising from the interaction between light and matter are typically very weak. Therefore, the enhancement of nonlinear signals has become a central pursuit in nonlinear optics. In traditional nonlinear optics, harmonic generation is typically enhanced through phase matching conditions \cite{Levine1975, Piltch1976}, increased incident light intensity \cite{Sarma2019}, or the use of natural materials with high nonlinear coefficients \cite{Barakat2013, Dutto2013}. However, these requirements are often stringent and difficult to achieve, which has historically limited the extent of harmonic generation enhancement.
In recent years, with the rapid development of surface plasmon resonance technology \cite{Ando2016, Deng2019, Deng2020, Deng2021, Deng2022, Blasten2020, Lei2025}, plasmon-enhanced nonlinear harmonic generation has emerged as a powerful strategy to overcome the inherent inefficiency of frequency conversion processes in the field of nanophotonics. The underlying principle is that when light irradiates a metal nanostructure, the conduction electrons collectively oscillate, giving rise to surface plasmon resonance and generating a strong localized electromagnetic field \cite{Danckwerts20017}, which in turn enhances the otherwise weak nonlinear signal \cite{Kim2011}. A key advantage of using surface plasmon resonance in metal nanostructures to enhance nonlinear optics is that the resonant structures are small in size, which allows the phase-matching conditions to be disregarded. This approach relies primarily on the unique electromagnetic resonance characteristics of surface plasmons to boost nonlinear effects \cite{Chen2012}.

Among all nonlinear optical phenomena, second harmonic generation has received extensive attention as the most fundamental process and for its significant applications in sensing and imaging \cite{Borcea2017, Assylbekov2021, Ren2024, Cakoni2025}. It is well known that second harmonic generation is forbidden within the electric dipole approximation in centrosymmetric bulk materials, such as the noble metals gold and silver \cite{Agarwal1981, Bloembergen1968, Dadap1999, Dadap2004}. Therefore, second harmonic generation is locally emitted from the surface where the centrosymmetry is broken \cite{Finazzi2007, Gennaro2016, Frizyuk2019}. The total second harmonic generation emission of the nanostructure results from the coherent addition of these local second harmonic generation emitters. If these structures also possess centrosymmetric shapes, the surface second harmonic generation response will again vanish in the electric dipole approximation because of the coherent addition. Consequently, the overall shape of the nanostructure plays a significant role in engineering an efficient second harmonic generation response. Numerous studies on second harmonic generation have focused on noble metal spherical nanoparticles, investigating the role of morphological deviations from a perfect spherical shape \cite{Hubert2007, Bachelier2008, Nappa2005a, Nappa2005b}. Additionally, it has also been noted \cite{Guyot1988} that there are nonvanishing electric quadrupolar and magnetic dipolar bulk contributions to the nonlinear polarization, which are related to the field gradient not included in the electric dipole approximation. Although these bulk contributions are typically small, for centrosymmetric bulk materials they can still be quite significant in comparison with the electric dipole contribution from the surface and therefore cannot be neglected. To comprehensively study and calculate second harmonic generation, several models beyond the dipole approximation have been proposed in the literature, such as the phenomenological model \cite{Sipe1987, Ponath1991}, the hydrodynamic model \cite{Corvi1986, Ciraci2012}, the jellium model \cite{Weber1987, Liebsch1989}, and the dipolium model \cite{Schaich1992, Mendoza1996}.

In this article, we develop a quantitative and rigorous mathematical theory that offers new perspectives on nonlinear harmonic generation in plasmonic structures arising from symmetry breaking. Although we primarily focus on second harmonic generation, we emphasize that our theoretical framework can be readily extended to higher-order harmonics. We investigate the plasmonic system in the static regime using a columnar nanowire with \(n\)-fold rotational symmetry (\(n \in \mathbb{N}\)), noting that a single metallic nanowire with threefold symmetry was previously studied in \cite{Singla2019}. To understand the nonlinear response of the nanowire, we construct a phenomenological model in which second harmonic generation originates from nonlinear sources confined to a thin selvedge region near the surface. This model naturally incorporates both the symmetry and the spatial extent of the nanowire, enabling the introduction of group theory to quantify the second harmonic generation in terms of symmetry degrees. We further develop an analytical approach to precisely characterize the second harmonic generation via multipolar contributions. Our theory complements existing physical descriptions of this practically important phenomenon and provides a rigorous account of how nonlinear optical efficiency depends on shape, size, symmetry, and defects in plasmonic structures.

The remainder of the paper is organized as follows. In Section \ref{sec:problem}, we present the mathematical formulation of the problem. Section \ref{sec:pre} provides auxiliary results from group theory and establishes the well-posedness of the governing equations. In Section \ref{sec:sensitivity-analysis}, we rigorously derive the asymptotic expansion of the perturbed linear and second harmonic electric fields. Section \ref{sec:proof} contains the main results for nonlinear harmonic generation via geometric symmetry perturbations in a uniform background field. In Section \ref{sec:proof6}, we present the main results concerning nonlinear harmonic generation via geometric and background field symmetry perturbations under a non-uniform background field. Finally, Section \ref{sec:conclusion} concludes the paper with a discussion of relevant implications and future directions.

\section{Mathematical setting of the problem}\label{sec:problem}
We begin by considering a nanowire made of a centrosymmetric material, illuminated by a monochromatic laser beam $\bm E_{\omega}^{\text{in}}$ with frequency $\omega$. The beam is polarized in the plane of the cross-section. The nanowire is modeled as an isolated, infinitely long cylinder placed in vacuum, with translational symmetry along its axis (the $\hat{\bm x}_3$ direction). We first consider, under a uniform background field, a cross-sectional geometry that lacks inversion symmetry, being a slight deformation from a centrosymmetric shape, as shown in Figure \ref{fig:schematic}(a). Moreover, we assume that the characteristic length scale $r_0$ of the cylinder (to which the deformation is applied) is much smaller than the wavelength of the incident electromagnetic field, which allows us to ignore the effects of retardation. Under this assumption, the monochromatic laser beam is reduced to an external electric field $\bm E_{\text{ext}}$ oscillating within the $x_1-x_2$ plane, as shown in Figure \ref{fig:schematic}(b). Therefore, we can address this problem within the nonretarded (quasi-static) regime, and then study the influence of the degree of symmetry of the material geometry on the multipole radiation of the second harmonic field. Secondly, in the case of a non-uniform background field, by perturbing both the background field and the boundary geometry, we induce a corresponding symmetry breaking and a change in the degree of symmetry, and investigate the profound connection between these factors and the dipole and multipole radiation behaviors of the second harmonic field.

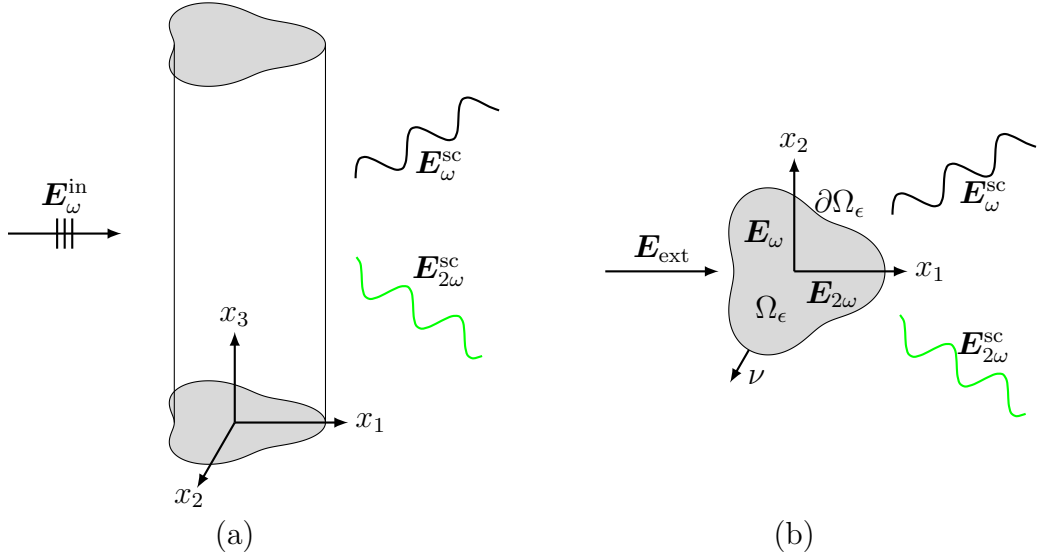
\begin{figure}
  \centering
\begin{tikzpicture}
  \def\R{1}          
  \def\a{0.2}        
  \def\k{3}          
  \def\h{5}         
  \def\vy{0.5}     
  \draw[fill=gray!30, domain=0:360, samples=200, smooth, variable=\t]
    plot ({ (\R+\a*cos(\k*\t)) * cos(\t) }, { (\R+\a*cos(\k*\t)) * sin(\t) * \vy });
  \draw[fill=gray!30, domain=0:360, samples=200, smooth, variable=\t]
    plot ({ (\R+\a*cos(\k*\t)) * cos(\t) }, { (\R+\a*cos(\k*\t)) * sin(\t) * \vy + \h });
  \draw (-\R+\a,0) -- (-\R+\a,\h);
  \draw (\R+\a,0) -- (\R+\a,\h);
  \draw[-latex,thick] (0,0)   -- (1.5,0);
      \node at (1.8, 0) {$x_1$};
  \draw[-latex,thick] (0,0)   -- (0,1.2);
        \node at (0, 1.4) {$x_3$};
  \draw[-latex,thick,rotate=-120,] (0,0)   -- (1,0);
          \node at (-0.6, -1) {$x_2$};
\draw[-latex,thick] (-3,2.5)   -- (-1.5,2.5);
  \draw[thick] (-2.25,2.7) -- (-2.25,2.3);   
    \draw[thick] (-2.15,2.7) -- (-2.15,2.3);   
        \draw[thick] (-2.35,2.7) -- (-2.35,2.3);   
          \node at (-2.25,3) {$\bm E_{\omega}^{\text{in}}$};
          \draw[thick, domain=0:6.28/3, smooth, variable=\x,rotate=30,shift={(3,2)}] 
    plot ({\x}, {0.2*sin(8*\x r)});
              \node at (2.7,3.4) {$\bm E_{\omega}^{\text{sc}}$};
          \draw[thick, green, domain=0:6.28/3, smooth, variable=\x,rotate=-30,shift={(0.3,2.5)}] 
    plot ({\x}, {0.2*cos(8*\x r)});
                  \node at (2.7,2) {$\bm E_{2\omega}^{\text{sc}}$};
    \node at (0, -1.5) {(a)};
\end{tikzpicture}
\qquad \quad
\begin{tikzpicture}
  \draw[fill=gray!30, domain=0:360, samples=200, smooth] 
    plot ({\x}: {1 + 0.2*cos(3*\x)});  
      \node at (-0.3, -0.5) {$\Omega_\epsilon$};     
      \node at (0.6, 0.9) {$\p \Omega_\epsilon$};
      \node at (-0.4, 0.5) {$\bm E_{\omega}$};
      \node at (0.5, -0.3) {$\bm E_{2\omega}$};
  \draw[-latex,thick] (0,0)   -- (1.5,0);
      \node at (1.8, 0) {$x_1$};
  \draw[-latex,thick] (0,0)   -- (0,1.5);
        \node at (0, 1.7) {$x_2$};
\draw[-latex,thick] (-2.5,0)   -- (-1,0);
\node at (-1.75,0.3) {$\bm E_{\text{ext}}$};
          \draw[thick, domain=0:6.28/3, smooth, variable=\x,rotate=30,shift={(1.5,0)}] 
    plot ({\x}, {0.2*sin(8*\x r)});
              \node at (2.5,1) {$\bm E_{\omega}^{\text{sc}}$};
          \draw[thick, green, domain=0:6.28/3, smooth, variable=\x,rotate=-30,shift={(1.5,0)}] 
    plot ({\x}, {0.2*cos(8*\x r)});
                  \node at (2.5,-1) {$\bm E_{2\omega}^{\text{sc}}$};
                  \draw[-latex,thick,rotate=-120,shift={(1.2,0)}] (0,0)   -- (0.5,0);
                        \node at (-0.5, -1.4) {$\nu$};
  \node at (0, -3.5) {(b)};
\end{tikzpicture}
  \caption{Schematic illustration of second harmonic generation, showing the relevant physical and geometric parameters. (a) Three-dimensional model: scattering of a monochromatic time-harmonic incident wave by a nanowire with the shape of a slightly deformed cylinder. (b) The reduced two-dimensional problem.}
\label{fig:schematic}
\end{figure}

As the inversion symmetry of the material is locally lost in a thin selvedge region around the surface, a nonlinear polarization is induced at the surface of the particle, which we write as
\begin{align}\label{eq:surface-nonlinear-polarization}
P^s_{i} = \sum_{jk} \chi_{ijk}^s F_{j}F_{k},
\end{align}
where $\chi_{ijk}^s$ are the components of the local nonlinear surface susceptibility. The field $\bm F$ is defined using quantities that are continuous across the surface, thereby avoiding ambiguity regarding the position within the selvedge at which the fields are evaluated. Specifically, $\bm F$ is composed of the normal projection of the displacement field and the parallel projection of the electric field, both evaluated at the surface. Thus,
\begin{align}\label{eq:F}
 \bm F = \bm E_{\omega}|_{+} = \varepsilon_\omega \bm E_{\omega}^{\perp}|_{-} +  \bm E_{\omega}^{\parallel}|_{-},
\end{align}
where the linear electric field $\bm E_{\omega} $ can be expressed through an electrostatic potential $u_\omega$, i.e., $\bm E_{\omega} = -\nabla u_\omega $. The notations $\perp$ and $\parallel$ denote the projections normal and parallel to the surface, respectively. The $|_{-}$ and $|_{+}$ subscripts indicate that the nonlinear polarization sheet is located just above the metal, while the fundamental electric field is evaluated just below the interface.

We assume that the thickness of the selvedge region is much smaller than the radius of the cylinder, and thus that the surface can be considered as locally flat. We further assume local invariance under rotations around the surface normal. Unless the surface contains structural features with intrinsic chirality, the metal-background interface possesses an isotropic mirror-symmetry plane perpendicular to the surface. Under these circumstances, the surface nonlinear susceptibility $\chi_{ijk}^s$ has only three independent components, namely, $\chi_{\perp\perp\perp}^s$, $\chi^s_{\perp\parallel\parallel}$, and $\chi_{\parallel\parallel\perp}^s = \chi_{\parallel\perp\parallel}^s$, where the symbols $\perp$ and $\parallel$ refer to the directions normal and tangent to the surface, respectively.
The nonlinear polarization induced on the surface of the cylinder is obtained by using \eqref{eq:surface-nonlinear-polarization} and \eqref{eq:F}. Specifically, its perpendicular component is
\begin{align}\label{eq:Ps-pre}
P^s_{\perp} = \chi_{\perp\perp\perp} ^s F_{\perp}F_{\perp}
= \chi_{\perp\perp\perp} ^s \left(\varepsilon_\omega\frac{\p u_\omega}{\p \nu}\Big|_{-}\right)^2,
\end{align}
and the tangential component is
\begin{align}\label{eq:Ps-par}
P^s_{\parallel} = \chi_{\parallel\parallel\perp} ^s F_{\parallel}F_{\perp} + \chi_{\parallel\perp\parallel} ^s F_{\perp}F_{\parallel}
= 2 \chi_{\parallel\perp\parallel} ^s F_{\perp} F_{\parallel}
= 2 \chi^s_{\parallel\perp\parallel} \left(\varepsilon_\omega\frac{\p u_\omega}{\p \nu}\Big|_{-}\right) \left(\frac{\p u_\omega}{\p T}\Big|_{-}\right),
\end{align}
where $\nu$ and $T$ denote the outward unit normal and tangential vectors to the surface, respectively. For convenience, we denote $\chi^s_{\perp}=\chi^s_{\perp\perp\perp}$ and $\chi^s_{\parallel} =\chi^s_{\parallel\perp\parallel}$ in subsequent sections.

The variation of the tangential component of the nonlinear surface polarization along the surface yields another contribution to the surface charge $\sigma^s$ beyond that due to the
termination of the bulk nonlinear polarization $\sigma^b$, where we use the superscript $s$ to denote its surface origin. It is given by
\begin{align}\label{eq:sigma-s}
  \sigma^s =& - \nabla_{\parallel} \cdot \bm P^s_{\parallel},
\end{align}
where $\nabla_{\parallel}$ is the gradient operator projected along the surface, and $\bm P_{\parallel}^s = P^s_{\parallel} T$ is the projection of $\bm P^s$ along the surface. 

To mathematically state the problem, let $\Omega$ be a bounded domain in $ \mathbb{R}^2$. Throughout this paper, we assume that $\Omega$ is of class $C^2$.
Let $H(x)$ be the harmonic function in $\mathbb{R}^2$, which denotes the background electrostatic potential. The governing equations are formulated as follows:
\begin{align}\label{eq:governing-equation}
\begin{cases}
\ds \Delta u_\omega = 0  & \mbox{in } \mathbb{R}^2, \ms \\
\ds u_\omega|_{+} = u_\omega|_{-} & \mbox{on }  \partial \Omega,\ms\\
\ds \frac{\partial u_\omega}{\partial \nu}\Big|_{+}=\varepsilon_\omega\frac{\partial u_\omega}{\partial \nu}\Big|_{-}   & \mbox{on } \partial \Omega, \ms \\
\ds  u_\omega = H(x) + O\left(|x|^{-1} \right) & \mbox{as } |x|\rightarrow + \infty,\ms\\
\ds \Delta u_{2\omega}= 0 & \mbox{in }  \mathbb{R}^2, \ms\\
\ds u_{2\omega}|_{+} - u_{2\omega}|_{-} = 4\pi P^s_{\perp} & \mbox{on }  \partial \Omega,\ms\\
\ds \frac{\partial u_{2\omega}}{\partial \nu} \Big|_{+} - \varepsilon_{2\omega}\frac{\partial u_{2\omega}}{\partial \nu} \Big|_{-} = -4\pi\sigma^s & \mbox{on } \partial \Omega,\ms \\
\ds  u_{2\omega} = o(1) & \mbox{as } |x|\rightarrow +\infty,
\end{cases}
\end{align}
where $\frac{\partial}{\partial \nu}$ denotes the outward normal derivative, and we use the notation $\frac{\partial u_{\omega}}{\partial \nu}\big|_{\pm}$ defined by
$$
\frac{\partial u_{\omega}}{\partial \nu}\Big|_{\pm}(x):=\lim_{t\rightarrow 0^+}\langle \nabla u_\omega(x\pm t\nu(x)),\nu(x) \rangle, \ \ x\in \p \Omega,
$$
where $\nu$ is the outward unit normal vector to $\p \Omega$. The same notation is used for $\frac{\partial u_{2\omega}}{\partial \nu}\big|_{\pm}$.

Note that the decay condition of $u_{2\omega}$ is $o(1)$, which admits two possible asymptotic behaviors, $O\left(|x|^{-1}\right)$ and $O \left(|x|^{-2}\right)$, depending on the symmetry of the geometry $\Omega$. For a geometry with inversion symmetry, the decay condition of $u_{2\omega}$ must be $O \left(|x|^{-2}\right)$, meaning the scattered field is radiated in the form of a quadrupole. However, for a geometry without inversion symmetry, both decay modes are possible. In this paper, we primarily explore what symmetry the geometry and background field must possess for multipole radiation to occur, particularly for the generation of dipole radiation. Therefore, we introduce the following definitions of dipole radiation and multipole radiation. The definitions of symmetry for the geometry of the material boundary and the background field will be introduced in Sections \ref{sec:proof} and \ref{sec:proof6}.

%In the case of a uniform background field, i.e., $H(x) = - E x_1$ ($E>0$ is a fixed constant), 
%We are now in a position to introduce the definitions of dipole radiation, multipole radiation and symmetry breaking, which play a central role in this paper.
\begin{defn}
A solution $u_{2\omega}$ to the Laplace equation whose domain of definition
contains the exterior of some circle is called dipole radiation if it satisfies a finiteness condition
\begin{align*}
  u_{2\omega} = O\left(\frac{1}{|x|}\right) \quad \mbox{as } |x|\rightarrow \infty,
\end{align*}
uniformly for all directions. More generally, we refer to the second-harmonic field as a 
$2^m$-multipole radiation field if it satisfies the following decay condition:
\begin{align*}
  u_{2\omega} = O\left(\frac{1}{|x|^m}\right) \quad \mbox{as } |x|\rightarrow \infty.
\end{align*}
\end{defn}

\section{Preliminaries and auxiliary results}\label{sec:pre}
In this section, we first collect some preliminary knowledge on group theory and then establish the representation formula of the solution of the governing equations.
\subsection{Group theory and symmetry degree}\label{sec:group-theory}
In mathematics, a dihedral group is the group of symmetries of a regular polygon, encompassing both its rotations and reflections. As one of the simplest structures in finite group theory, dihedral groups serve as fundamental examples and play a crucial role in both group theory and geometry.

Let \(n\) be a positive integer. The dihedral group \(D_n\) is defined by
\[
D_n = \{ e, R, R^2, \dots, R^{n-1}, S, SR, SR^2, \dots, SR^{n-1} \},
\]
where \(e\) is the identity, \(R\) represents a counterclockwise rotation of \(2\pi/n\) about the geometric center, and \(S\) denotes a reflection about a line of symmetry. It can be generated by the symbols \(R\) and \(S\) satisfying the following relations:
\[
R^n = S^2 = e, \quad SRS = R^{-1}.
\]
The order of \(D_n\), denoted by \(|D_n|\), is the number of its elements. Hence, \(|D_n| = 2n\).

We define the center of \(D_n\), denoted by \(Z(D_n)\), as the set of elements that commute with all elements of \(D_n\):
\[
Z(D_n) = \{ z \in D_n \mid  az = za, \ \forall a \in D_n\}.
\]
From this definition, one finds that \(Z(D_1) = \{e, S\}\), \(Z(D_2) = \{e, R, S, RS\}\), and for \(n \geq 3\), \(Z(D_n) = \{ e \}\) if \(n\) is odd, while \(Z(D_n) = \{ e, R^{n/2} \}\) if \(n\) is even.
Note that \(D_n\) is abelian if and only if its center is the whole group, i.e., \(Z(D_n) = D_n\), which implies that rotations and reflections commute. Hence, we have the following lemma.

\begin{lem}
Dihedral groups \(D_n\) are non-abelian for integers \(n \geq 3\). In particular, for \(n = 1, 2\), \(D_n\) is abelian.
\end{lem}

Since the dihedral group \(D_n\) is the group of symmetries for a regular polygon with $n$ sides, it is common to restrict to $n\geq3$.  We next introduce the following definition and theory to characterize geometric invariance under the dihedral group action.
\begin{defn}\label{def:grouinvar}
A two-dimensional closed curve \(\Gamma\) (or the closed region it encloses) is said to be invariant under the dihedral group \(D_n\) if for every \(d \in D_n\) and every \(x \in \Gamma\), we have \(dx \in \Gamma\), i.e., \(D_n(\Gamma) = \Gamma\).
\end{defn}

\begin{thm}\label{thm:dihgrouinvar}
Let the boundary $\partial \Omega$ be defined by (\ref{eq:geometry}) with $f =r_0 \cos(n\theta)$, and \(D_q\) is $2q$ order dihedral groups. If $q|n$, then $\partial \Omega$ is invariant under \(D_q\). Furthermore, the symmetry degree of $\partial \Omega$ is $2n$.
\end{thm}
\begin{proof}
Since $q|n$, there exists $k\in\mathbb{Z_+}$ so that $n=kq$. Set $z=r_0(1+\epsilon\cos(n\theta))e^{i\theta}\in\partial\Omega$, we see 
\begin{align*}
Rz&= \e^{i\frac{2\pi}{q}}r_0(1+\epsilon\cos(n\theta))\e^{i\theta}\\
&=r_0\left(1+\epsilon\cos\left(n\left(\theta+\frac{2k\pi}{n}\right)\right)\right)\e^{i\left(\theta+\frac{2k\pi}{n}\right)}.
\end{align*}
Hence, $Rz\in\partial\Omega$, i.e., $\partial\Omega$ is invariant under rotational transformation. Furthermore, notice $\cos(n(-\theta))=\cos(n\theta)$, it deduce $Sz\in\partial\Omega$. Since dihedral groups $D_q$ is a cyclic group, in light of definition \ref{def:grouinvar}, we imply $\partial \Omega$ is invariant under \(D_q\). Clearly, From the definition of symmetry degree (definition \ref{defn:symmetry-degree}), we find the symmetry degree of $\partial \Omega$ is $2n$.
\end{proof}

\begin{rem}
From Theorem \ref{thm:dihgrouinvar}, if we restrict attention to dihedral groups with $n\geq3$, it is clear that, at $n\geq3$ the boundary geometry attains the lowest degree of symmetry. Conversely, the larger the value of \(n\), the smaller the rotation angle \(2\pi/n\) becomes, meaning that only a small rotation is required to map the shape onto itself—i.e., the shape is more symmetric. Consequently, the order \(2n\) of the dihedral group serves as an indicator of the degree of symmetry in the perturbed geometry: the larger \(n\) is, the stronger the geometric symmetry.
\end{rem}

To investigate the effect of the background field’s symmetry on nonlinear optical responses, we introduce the following definition.

\begin{defn}\label{defn:q-order}
We say that a bivariate real harmonic polynomials \( H(x) \) has \( q \)-th order symmetry if \( H(x) \) is invariant under the \( q \)-th order binary group \( D_q \), i.e., for all \( x \in \mathbb{R}^2 \),
\[
H(gx) = H(x)
\]
where \( g \in D_q \). The larger \( q \) is, the higher or stronger the symmetry of \( H(x) \) is considered to be.
\end{defn}

The following theorem is useful in the analysis in Section \ref{sec:Sb-non-uniform}.
\begin{thm}\label{thm:q-order}
A real harmonic polynomial $H$ of degree $\ell$ in two variables has \( q \)-th order symmetry if and only if $q \mid \ell$.
\end{thm}

\begin{proof}
We first prove the sufficiency. If $q \mid \ell$, then there exists an integer $L$ such that $\ell = qL$. For a rotation $R \in D_q$, we have
\[
H(Rx) = c \cdot \operatorname{Re}\left((Rz)^\ell\right) = c \cdot \operatorname{Re}\left(\left(\e^{\i\frac{2\pi}{q}}z\right)^\ell\right)
= c \cdot \operatorname{Re}\left(\e^{\i L (2\pi)} z^\ell\right)  = c \cdot \operatorname{Re}(z^\ell) = H(x),
\]
so the degree-$\ell$ real harmonic polynomial is invariant under rotation. For a reflection $S$, we have
\[
H(Sx) = c \cdot \operatorname{Re}(\bar{z}^\ell) = c \cdot \operatorname{Re}(z^\ell) = H(x),
\]
i.e., it is also invariant under reflection. Hence sufficiency holds.

Conversely, to prove necessity, assume $H$ has \( \ell \)-th order symmetry.  Since
\[
H(x) = c \cdot \operatorname{Re}(z^\ell) = c \cdot \frac{z^\ell + \overline{z}^\ell}{2},
\]
by Definition \ref{defn:q-order} of \(\ell\)-th order symmetry, for a rotation $R \in D_q$ we must have $H(Rx) = H(x)$, i.e.,
\[
\frac{c}{2}\left(\e^{\i \ell \frac{2\pi}{q}}z^\ell + \e^{-\i \ell\frac{2\pi}{q}}\overline{z}^\ell\right) = \frac{c}{2}\left(z^\ell + \overline{z}^\ell \right).
\]
Thus
\[
\e^{\i \ell \frac{2\pi}{q}}z^\ell + \e^{-\i \ell\frac{2\pi}{q}}\overline{z}^\ell = z^\ell + \overline{z}^\ell.
\]
Noting that $z^\ell$ and $\overline{z}^\ell$ are linearly independent, we obtain
\[
\e^{\i\frac{\ell}{q}2\pi} = \e^{-\i\frac{\ell}{q}2\pi} = 1,
\]
which implies $q \mid \ell$.
\end{proof}

\subsection{Layer potentials formulation and well-posedness}
%\label{sec-layer-potentials}
For the domain $\Omega$, let us now introduce the single-layer potential by
\begin{align*}
\mathcal{S}_\Omega[\vartheta](x) :=\int_{\partial \Omega}G(x-y)\vartheta(y)\d s(y), \quad  x\in \mathbb{R}^2,
\end{align*}
and
\begin{align*}
\mathcal{D}_\Omega[\vartheta](x) :=\int_{\partial \Omega} \frac{\partial G(x-y)}{\partial \nu(y)}\vartheta(y)\d s(y), \quad  x\in \mathbb{R}^2\setminus \partial \Omega,
\end{align*}
where $\vartheta\in L^2(\p \Omega)$ is the density function, and the Green function $G(x-y)$ for the Laplacian in $\mathbb{R}^2$ is given by
\begin{align*}
G(x-y)=\frac{1}{2\pi}\ln|x-y|.
\end{align*}

Then the following jump relations hold:
\begin{align}
\label{jump-relation-S}
\frac{\partial\mathcal{S}_\Omega[\vartheta]}{\partial \nu}\bigg|_{\pm}(x)&=\left(\pm\frac{1}{2} \mathcal{I} +\mathcal{K}^*_\Omega \right)[\vartheta](x), \ \ x\in \partial \Omega, \ms\\
\mathcal{D}_\Omega[\vartheta]|_{\pm}(x)&=\left(\mp\frac{1}{2} \mathcal{I} + \mathcal{K}_\Omega \right)[\vartheta](x), \ \ x\in \partial \Omega, \label{jump-relation-D}
\end{align}
where $\mathcal{K}_{\Omega}$ is the boundary integral operator defined by
\begin{align*}
\mathcal{K}_\Omega[\vartheta](x)=\int_{\partial \Omega}\frac{\partial G(x-y)}{\partial \nu (y)}\vartheta(y)\d s(y),
\end{align*}
and $\mathcal{K}^*_{\Omega}$ is the $L^2$-adjoint of $\mathcal{K}_{\Omega}$, i.e.,
\begin{align*}
\mathcal{K}_\Omega^*[\vartheta](x)=\int_{\partial \Omega}\frac{\partial G(x-y)}{\partial \nu (x)}\vartheta(y)\d s(y).
\end{align*}

In order to derive the representation formulas for the solutions to the governing equation \eqref{eq:governing-equation}, we make use of the following lemma \cite{Ammari2007}.

\begin{lem}\label{I+K-invertible}
The operator $\lambda \mathcal{I} -\mathcal{K}^*_{\Omega}: L_0^2(\partial \Omega)\rightarrow L_0^2(\partial \Omega)$ is invertible. Here $L_0^2 := \{u\in L^2(\partial \Omega)\mid \int_{\partial \Omega} u \d s=0\}$.
\end{lem}

Using the layer potential theory and Lemma \ref{I+K-invertible}, we obtain the following theorem.

\begin{thm}\label{well-posedness}
Let $u_\omega$, $u_{2\omega}\in C^2(\mathbb{R}^2)\cap C(\mathbb{R}^2)$ be the classical solutions of \eqref{eq:governing-equation}.
Then $u_\omega$ can be represented as
\begin{equation}\label{eq:sol-u1}
u_\omega = H(x) + \mathcal{S}_{\Omega}[\varphi](x),\quad x\in\mathbb{R}^2,
\end{equation}
where the density function $\varphi \in L_0^2(\p \Omega)$ satisfies
\begin{align}\label{eq:u1-density}
\left(\lambda_{\omega} \mathcal{I} -\mathcal{K}^*_{\Omega}\right)[\varphi]=\frac{\partial H}{\partial \nu} \quad \mbox{on }\p \Omega,
\end{align}
with $\lambda_{\omega}$ given by
\begin{equation}\label{eq:lambda-omega}
 \lambda_\omega = \frac{\varepsilon_\omega+1}{2(\varepsilon_\omega-1)}.
\end{equation}

Moreover, $u_{2\omega}$ can be represented using the double-layer potential $\mathcal{D}_{\Omega}$ and the single-layer potential $\mathcal{S}_{\Omega}$ as follows:
 \begin{equation}\label{eq:sol-u2}
u_{2\omega} = \mathcal{D}_{\Omega}[\phi](x) + \mathcal{S}_{\Omega}[\psi](x),\quad  x\in  \mathbb{R}^2,
\end{equation}
where the pair $(\phi, \psi)\in L_0^2(\p \Omega)\times L_0^2(\p \Omega)$ satisfies
\begin{align}
\label{eq:u2-density}
\begin{cases}
\phi = -4\pi P^s_{\perp} &\quad \mbox{on }\p \Omega, \ms\\
\ds (\lambda_{2\omega} \mathcal{I} -\mathcal{K}^*_{\Omega})[\psi] - \frac{\p \mathcal{D}_{\Omega}}{\p \nu}[\phi] = - 4\pi\frac{\sigma^s + \sigma^b}{\varepsilon_{2\omega}-1} &\quad \mbox{on }\p \Omega. 
\end{cases}
\end{align}
Here $\lambda_{2\omega}$ is given by
\begin{equation}\label{eq:lambda-2omega}
 \lambda_{2\omega} = \frac{\varepsilon_{2\omega}+1}{2(\varepsilon_{2\omega}-1)}.
\end{equation}
Furthermore, there exists a constant $C$ such that
\begin{align}
\label{stability}
\|\phi\|_{L^2(\p \Omega)}+\|\psi\|_{L^2(\p \Omega)}
\leq C \left\|\nabla H \right\|_{L^2(\p \Omega)}.
\end{align}
\end{thm}

\begin{proof}
Applying the jump formula \eqref{jump-relation-S} for the normal derivative of the single-layer potential, the boundary condition on $\partial\Omega$ satisfied by \eqref{eq:sol-u1} becomes \eqref{eq:u1-density}. By Lemma \ref{I+K-invertible}, this yields a unique density $\varphi$. Similarly, imposing the transmission conditions along $\partial\Omega$ on \eqref{eq:sol-u2} gives \eqref{eq:u2-density}. The density $\phi$ is already expressed in terms of $-4\pi P^s_{\perp}$ using the jump formula \eqref{jump-relation-D}, and therefore it exists and is unique. Lemma \ref{I+K-invertible} also ensures the unique existence of $\psi$. The stability estimate \eqref{stability} then follows from the solvability and the closed graph theorem.

The proof is complete.
\end{proof}

Hence, by Theorem \ref{well-posedness}, the governing equation \eqref{eq:governing-equation} is well-posed.

\section{Asymptotic expansions}\label{sec:sensitivity-analysis}
In this section, we consider geometric symmetry breaking via shape perturbation.  We rigorously derive the asymptotic expansions of the perturbed linear and second harmonic electric fields using the field expansion method \cite{Nicholls2003}, thereby obtaining the leading-order and first-order coupled systems. The representation formulas for the solutions to these systems are also expressed in terms of the layer potential.

%Let $D$ be a bounded domain in $ \mathbb{R}^2$. 
For small $\epsilon \in \mathbb{R}_{+}$, we let $\partial \Omega$ be an $\epsilon$-perturbation of $\Omega$, i.e.,
\begin{align}\label{eq:geometry}
\partial \Omega_\epsilon = \{ x + \epsilon f(x) \nu(x) : x \in \partial \Omega \},
\end{align}
where $\nu$ is the outward unit normal to $\partial \Omega$, and $f$ is an arbitrary smooth function. 

Let $\tilde{\nu}$ and $\widetilde{T}$ be the outward unit normal vector and the unit tangential vector on $\partial \Omega$, respectively. The following expansions of $\tilde{\nu}$ and $\widetilde{T}$ hold \cite{M. Lim2012}:
\begin{align}
\label{eq:nu-expansion}
\tilde{\nu} (\tilde{x})&=\nu(x)-\epsilon f'(x)T(x)+ O\left(\epsilon^2\right),\\
\widetilde{T} (\tilde{x})&= T(x)+\epsilon f'(x)\nu(x)+O\left(\epsilon^2\right). \label{eq:T-expansion}
\end{align}
Here and throughout this paper, $T$ is the unit tangential vector on $\partial \Omega$ and $f'(x)$ is the tangential derivative of $f$ on $\partial \Omega$, i.e., $f'=\frac{\partial f}{\partial T}$.

\subsection{Formal derivations via the field expansion method}
%In this subsection, we formally prove Theorem \ref{thm-expansion} based on the FE method.
We first derive the asymptotic expansion of $u_\omega^\epsilon$, solution to \eqref{eq:governing-equation} with $\Omega_\epsilon$, as $\epsilon$ goes to zero. We start by expanding $u^\epsilon_\omega$ in powers of $\epsilon$, that is
\begin{align}
\label{eq:FF-expansion}
u_\omega^{\epsilon} = u_\omega^{(0)} + \epsilon u_\omega^{(1)} + O\left(\epsilon^2\right),
\end{align}
where $u_\omega^{(n)}$, $n = 0, 1$, are well-defined in $\mathbb{R}^2\setminus\overline{\Omega}$, and satisfy
\begin{equation*}
\begin{cases}
\ds \Delta u_\omega^{(n)} = 0 & \mbox{in } \mathbb{R}^2, \ms \\
\ds u_\omega^{(n)} = \delta_{0,n} H(x) + O\left(|x|^{-1}\right) & \mbox{as } |x| \to +\infty,
\end{cases}
\end{equation*}
with $\delta_{0,n}$ the Kronecker symbol.

For $x\in\partial \Omega$, let $\tilde{x}=x+\epsilon f(x)\nu(x) \in \partial \Omega_\epsilon$. Then we have the following Taylor expansions:
\begin{align}
u^\epsilon_\omega|_{\pm}(\tilde{x}) &= u_\omega^{(0)}|_{\pm}(x) + \epsilon u_\omega^{(1)}|_{\pm}(x) + \epsilon f \frac{\partial u_\omega^{(0)}}{\partial \nu}\Big|_{\pm}(x) + O\left(\epsilon^2\right), \quad x \in \partial \Omega.\label{eq:u-omega-inner-outer}
\end{align}
The normal derivative $\frac{\partial u_\omega^\epsilon}{\partial \tilde{\nu}}(\tilde{x})$ and tangential derivative $\frac{\partial u_\omega^\epsilon}{\partial \widetilde{T}}(\tilde{x})$ on $\partial \Omega_\epsilon$ are given by
\begin{align}\label{normal-derivative-Omega}
\frac{\partial u_\omega^\epsilon}{\partial \tilde{\nu}}(\tilde{x}) = \nabla u_\omega^\epsilon (\tilde{x}) \cdot \tilde{\nu}(\tilde{x}),\\
\frac{\partial u_\omega^\epsilon}{\partial \widetilde{T}}(\tilde{x}) = \nabla u_\omega^\epsilon (\tilde{x}) \cdot \widetilde{T}(\tilde{x}),\label{tangential-derivative-Omega}
\end{align}
where $\tilde{\nu}(\tilde{x})$ and $\widetilde{T}(\tilde{x})$ are defined by \eqref{eq:nu-expansion} and \eqref{eq:T-expansion}, respectively.
To evaluate $\nabla u_\omega^\epsilon (\tilde{x})$ appearing in \eqref{normal-derivative-Omega} and \eqref{tangential-derivative-Omega}, we expand $\nabla u_\omega^\epsilon$ around $\partial \Omega$ and use \eqref{eq:FF-expansion} to obtain
\begin{align}\label{varphi-gradient}
\nabla u_\omega^\epsilon (\tilde{x}) = \nabla u_\omega^{(0)}(x) + \epsilon \nabla u_\omega^{(1)}(x) + \epsilon f \nabla^2 u_\omega^{(0)} \nu(x) + O\left(\epsilon^2\right), \quad x\in \partial \Omega.
\end{align}
It then follows from \eqref{eq:nu-expansion}, \eqref{eq:T-expansion}, \eqref{normal-derivative-Omega}, \eqref{tangential-derivative-Omega} and \eqref{varphi-gradient} that
\begin{align}
\frac{\partial u_\omega^\epsilon}{\partial \tilde{\nu}}\Big|_{\pm} (\tilde{x})
=\frac{\partial u_\omega^{(0)}}{\partial \nu} \Big|_{\pm} (x) + \epsilon \left(\frac{\partial u_\omega^{(1)}}{\partial \nu}\Big|_{\pm} (x) + f \frac{\partial^2 u_\omega^{(0)}}{\partial \nu^2}\Big|_{\pm}(x) - f'\frac{\partial u_\omega^{(0)}}{\partial T} \Big|_{\pm} (x) \right) + O\left(\epsilon^2\right), &\quad x\in \partial \Omega,\label{eq:normal-derivative-Omega-expansion-inner-outer}\\
\frac{\partial u_\omega^\epsilon}{\partial \widetilde{T}}\Big|_{-} (\tilde{x})
=\frac{\partial u_\omega^{(0)}}{\partial T} \Big|_{-} (x) + \epsilon \left(\frac{\partial u_\omega^{(1)}}{\partial T}\Big|_{-} (x) + f \frac{\partial}{\partial \nu}\left(\frac{\partial u_\omega^{(0)}}{\partial T}\right)\Big|_{-}(x) + f'\frac{\partial u_\omega^{(0)}}{\partial \nu} \Big|_{-} (x) \right) + O\left(\epsilon^2\right), &\quad x\in \partial \Omega.\label{eq:tangential-derivative-Omega-expansion}
\end{align}
By using the transmission conditions on $\partial \Omega_\epsilon$, we deduce from \eqref{eq:u-omega-inner-outer} and \eqref{eq:normal-derivative-Omega-expansion-inner-outer} that
\begin{align*}
u_\omega^{(0)}|_{+} = u_\omega^{(0)}|_{-}, & \quad x\in \partial \Omega,\\
\frac{\partial u^{(0)}_\omega}{\partial \nu}\Big|_{+} = \varepsilon_\omega\frac{\partial u^{(0)}_\omega}{\partial \nu}\Big|_{-}, & \quad x\in \partial \Omega,
\end{align*}
and
\begin{align*}
u_\omega^{(1)}|_{+}-u_\omega^{(1)}|_{-} = f\left(\frac{\partial u_\omega^{(0)}}{\partial \nu}\Big|_{-} - \frac{\partial u_\omega^{(0)}}{\partial \nu}\Big|_{+}\right), \quad &x\in \partial \Omega,\\
\frac{\partial u_\omega^{(1)}}{\partial \nu} \Big|_{+} - \varepsilon_\omega\frac{\partial u_\omega^{(1)}}{\partial \nu} \Big|_{-} = f \left(\varepsilon_\omega \frac{\partial^2 u_\omega^{(0)}}{\partial \nu^2}\Big|_{-} - \frac{\partial^2 u_\omega^{(0)}}{\partial \nu^2}\Big|_{+} \right) + f'\left(\frac{\partial u_\omega^{(0)}}{\partial T}\Big|_{+} - \varepsilon_\omega\frac{\partial u_\omega^{(0)}}{\partial T}\Big|_{-} \right), \quad &x\in \partial \Omega.
\end{align*}

To perform a Taylor expansion of the function $\sigma^s$ defined by \eqref{eq:sigma-s}, we parameterize the boundary $\partial \Omega_\epsilon$ as $X(t) = (x_1(t), x_2(t))$. We then have
\begin{align}\label{eq:sigma-s-parameterize}
\sigma^s = -\frac{1}{h} \frac{\d P^s_{\parallel}}{\d t},
\end{align}
where $h$ denotes the Lam\'e coefficient of the boundary curve $\partial \Omega_\epsilon$. By utilizing the perturbation properties of boundaries and knowledge of differential geometry, we have
\begin{align}\label{eq:Lame-expansion}
h = h^{(0)} + \epsilon h^{(0)} \kappa,
\end{align}
where $\kappa$ denotes the curvature of the boundary $\partial \Omega$.

In a similar way, we next expand $u_{2\omega}^\epsilon$, solution to \eqref{eq:governing-equation} with $\Omega_\epsilon$, in powers of $\epsilon$, that is
\begin{align*}
u_{2\omega}^\epsilon = u_{2\omega}^{(0)} + \epsilon u_{2\omega}^{(1)} + O\left(\epsilon^2\right),
\end{align*}
where $u_{2\omega}^{(n)}$, $n = 0, 1$, are well-defined in $\mathbb{R}^2$, and satisfy
\begin{equation*}
\begin{cases}
\ds \Delta u_{2\omega}^{(n)} = 0 & \mbox{in } \mathbb{R}^2, \ms \\
\ds u_{2\omega}^{(n)} = o(1) & \mbox{as } |x| \to +\infty,
\end{cases}
\end{equation*}
with $\delta_{0,n}$ the Kronecker symbol.

For $x\in\partial \Omega$, let $\tilde{x}=x+\epsilon f(x)\nu(x) \in \partial \Omega_\epsilon$. From \eqref{eq:Ps-pre}, \eqref{eq:Ps-par}, \eqref{eq:normal-derivative-Omega-expansion-inner-outer} and \eqref{eq:tangential-derivative-Omega-expansion}, it follows that
\begin{align}
P^{s}_{\perp}(\tilde{x})
&= \chi_{\perp}^s \varepsilon_\omega^2 \left(\frac{\partial u^{(0)}_\omega}{\partial \nu}\Big|_{-} (x)\right)^2 \nonumber \\
&\quad + 2\epsilon \chi_{\perp}^s \varepsilon_\omega^2 \left(\frac{\partial u^{(0)}_\omega}{\partial \nu}\Big|_{-}(x)\right) \left(\frac{\partial u_\omega^{(1)}}{\partial \nu}\Big|_{-} (x) + f \frac{\partial^2 u_\omega^{(0)}}{\partial \nu^2}\Big|_{-}(x) - f'\frac{\partial u_\omega^{(0)}}{\partial T} \Big|_{-} (x) \right) + O\left(\epsilon^2\right) \nonumber \\
&= P^{s,(0)}_{\perp} + \epsilon P^{s,(1)}_{\perp} + O\left(\epsilon^2\right)\label{eq:Ps-per-expansion}
\end{align}
and
\begin{align}
P^{s}_{\parallel}(\tilde{x})
&= 2 \chi_{\parallel}^s \varepsilon_\omega \left(\frac{\partial u^{(0)}_\omega}{\partial \nu}\Big|_{-}(x) \right) \left(\frac{\partial u^{(0)}_\omega}{\partial T}\Big|_{-}(x)\right) \nonumber \\
&\quad + 2\epsilon \chi_{\parallel}^s \varepsilon_\omega \left(\frac{\partial u^{(0)}_\omega}{\partial \nu}\Big|_{-}(x)\right) \left(\frac{\partial u_\omega^{(1)}}{\partial T}\Big|_{-} (x) + f \frac{\partial}{\partial \nu}\left(\frac{\partial u_\omega^{(0)}}{\partial T}\right)\Big|_{-}(x) + f'\frac{\partial u_\omega^{(0)}}{\partial \nu} \Big|_{-} (x) \right) \nonumber \\
&\quad + 2\epsilon \chi_{\parallel}^s \varepsilon_\omega \left(\frac{\partial u^{(0)}_\omega}{\partial T}\Big|_{-}(x)\right) \left(\frac{\partial u_\omega^{(1)}}{\partial \nu}\Big|_{-} (x) + f \frac{\partial^2 u_\omega^{(0)}}{\partial \nu^2}\Big|_{-}(x) - f'\frac{\partial u_\omega^{(0)}}{\partial T} \Big|_{-} (x) \right) + O\left(\epsilon^2\right) \nonumber \\
&= P^{s,(0)}_{\parallel} + \epsilon P^{s,(1)}_{\parallel} + O\left(\epsilon^2\right). \label{eq:Ps-par-expansion}
\end{align}
By \eqref{eq:sigma-s-parameterize}, \eqref{eq:Lame-expansion} and \eqref{eq:Ps-par-expansion}, we have
\begin{align}\label{eq:sigma-s-expansion}
\sigma^s = -\frac{1}{h^{(0)}} \frac{\d P^{s,(0)}_{\parallel}}{\d t} + \epsilon \left(\frac{\kappa}{h^{(0)}} \frac{\d P^{s,(0)}_{\parallel}}{\d t} - \frac{1}{h^{(0)}} \frac{\d P^{s,(1)}_{\parallel}}{\d t}  \right) + O\left(\epsilon^2\right) = \sigma^{s,(0)} + \epsilon \sigma^{s,(1)} + O\left(\epsilon^2\right).
\end{align}

Similar to \eqref{eq:u-omega-inner-outer} and \eqref{eq:normal-derivative-Omega-expansion-inner-outer}, we have the following Taylor expansions:
\begin{align}
u^\epsilon_{2\omega}|_{\pm}(\tilde{x}) &= u_{2\omega}^{(0)}|_{\pm}(x) + \epsilon u_{2\omega}^{(1)}|_{\pm}(x) + \epsilon f \frac{\partial u_{2\omega}^{(0)}}{\partial \nu}\Big|_{\pm}(x) + O\left(\epsilon^2\right), \quad x \in \partial \Omega.\label{eq:u-2omega-inner-outer}
\end{align}
and
\begin{align}
\frac{\partial u_{2\omega}^\epsilon}{\partial \nu}\Big|_{\pm} (\tilde{x})
=\frac{\partial u_{2\omega}^{(0)}}{\partial \nu} \Big|_{\pm} (x) + \epsilon \left(\frac{\partial u_{2\omega}^{(1)}}{\partial \nu}\Big|_{\pm} (x) + f \frac{\partial^2 u_{2\omega}^{(0)}}{\partial \nu^2}\Big|_{\pm}(x) - f'\frac{\partial u_{2\omega}^{(0)}}{\partial T} \Big|_{\pm} (x) \right) + O\left(\epsilon^2\right), &\quad x\in \partial \Omega.\label{eq:normal-derivative-2Omega-expansion-inner-outer}
\end{align}
From \eqref{eq:Ps-per-expansion}, \eqref{eq:sigma-s-expansion}, \eqref{eq:u-2omega-inner-outer} and \eqref{eq:normal-derivative-2Omega-expansion-inner-outer}, the transmission conditions on $\partial \Omega_\epsilon$ immediately yield
\begin{align*}
u^{(0)}_{2\omega}|_{+} - u^{(0)}_{2\omega}|_{-} = 4\pi P^{s,(0)}_{\perp}  &\quad \mbox{on } \partial \Omega,\\
\frac{\partial u^{(0)}_{2\omega}}{\partial \nu} \Big|_{+} - \varepsilon_{2\omega}\frac{\partial u^{(0)}_{2\omega}}{\partial \nu} \Big|_{-} = -4\pi \sigma^{s,(0)}  &\quad \mbox{on } \partial \Omega,
\end{align*}
and
\begin{align*}
u_{2\omega}^{(1)}|_{+} - u_{2\omega}^{(1)}|_{-} &= f\left(\frac{\partial u_{2\omega}^{(0)}}{\partial \nu}\Big|_{-} - \frac{\partial u_{2\omega}^{(0)}}{\partial \nu}\Big|_{+}\right) + 4\pi P^{s,(1)}_{\perp} && \mbox{on } \partial \Omega, \\
\frac{\partial u_{2\omega}^{(1)}}{\partial \nu} \Big|_{+} - \varepsilon_{2\omega}\frac{\partial u_{2\omega}^{(1)}}{\partial \nu} \Big|_{-} &= f \left(\varepsilon_{2\omega} \frac{\partial^2 u_{2\omega}^{(0)}}{\partial \nu^2}\Big|_{-} - \frac{\partial^2 u_{2\omega}^{(0)}}{\partial \nu^2}\Big|_{+} \right) 
 + f'\left(\frac{\partial u_{2\omega}^{(0)}}{\partial T}\Big|_{+} - \varepsilon_{2\omega}\frac{\partial u_{2\omega}^{(0)}}{\partial T}\Big|_{-} \right) \nonumber \\
&\quad - 4\pi \sigma^{s,(1)}  && \mbox{on } \partial \Omega.
\end{align*}

Summarizing the above results, we obtain the following theorem.
\begin{thm}\label{thm-expansion}
Let $u_\omega^\epsilon$ and $u_{2\omega}^\epsilon$ be the solutions to \eqref{eq:governing-equation} with $\Omega_\epsilon$. For $x\in  \mathbb{R}^2$, the following pointwise asymptotic expansions hold
\begin{align*}
  u^\epsilon_\omega(x) = u_\omega^{(0)}(x)+\epsilon u_\omega^{(1)}(x)+ O\left(\epsilon^2\right),
\end{align*}
and
\begin{align*}
  u^\epsilon_{2\omega}(x) =u_{2\omega}^{(0)}(x)+\epsilon u_{2\omega}^{(1)}(x)+ O\left(\epsilon^2\right),
\end{align*}
where the remainder $O(\epsilon^2)$ depends only on the $\mathcal{C}^2$-norm of $\partial \Omega$ and $\mathcal{C}^1$-norm of $f$. $u_\omega^{(0)}$ and $u_{2\omega}^{(0)}$ are the solutions to the following leading-order coupled system
\begin{align}\label{eq:leading-term-governing}
\begin{cases}
\ds \Delta u_\omega^{(0)} = 0  & \mbox{in } \mathbb{R}^2, \ms \\
\ds u_\omega^{(0)}|_{+} = u_\omega^{(0)}|_{-} & \mbox{on }  \partial \Omega,\ms\\
\ds \frac{\partial u^{(0)}_\omega}{\partial \nu}\Big|_{+}=\varepsilon_\omega\frac{\partial u^{(0)}_\omega}{\partial \nu}\Big|_{-}   & \mbox{on } \partial \Omega, \ms \\
\ds  u_\omega = H(x) + O\left(|x|^{-1}\right) & \mbox{as } |x|\rightarrow + \infty,\ms\\
\ds \Delta u^{(0)}_{2\omega}= 0 & \mbox{in }  \mathbb{R}^2, \ms\\
\ds u^{(0)}_{2\omega}|_{+} - u^{(0)}_{2\omega}|_{-} = 4\pi P^{s,(0)}_{\perp} & \mbox{on }  \partial \Omega,\ms\\
\ds \frac{\partial u^{(0)}_{2\omega}}{\partial \nu} \Big|_{+} - \varepsilon_{2\omega}\frac{\partial u^{(0)}_{2\omega}}{\partial \nu} \Big|_{-} = -4\pi \sigma^{s,(0)} & \mbox{on } \partial \Omega,\ms \\
\ds  u^{(0)}_{2\omega} = O\left(|x|^{-2}\right) & \mbox{as } |x|\rightarrow +\infty,
\end{cases}
\end{align} 
and the pair $(u_\omega^{(1)}, u_{2\omega}^{(1)})$ is the unique solution to the following first-order coupled system
\begin{align}\label{eq:first-order-equation}
\begin{cases}
\ds \Delta u_\omega^{(1)}= 0  \quad \ &\mbox{in} \  \mathbb{R}^2,\ms \\
\ds u_\omega^{(1)}|_{+}-u_\omega^{(1)}|_{-}=I_1 \quad &\mbox{on} \ \partial \Omega,\ms \\
\ds \frac{\partial u_\omega^{(1)}}{\partial \nu} \Big|_{+} - \varepsilon_\omega\frac{\partial u_\omega^{(1)}}{\partial \nu} \Big|_{-} = I_2 \quad &\mbox{on } \partial \Omega, \ms \\
\ds u_\omega^{(1)} = O\left(|x|^{-1}\right)\ \ &as\ |x|\rightarrow +\infty,\ms \\
\ds \Delta u^{(1)}_{2\omega}= 0 & \mbox{in }  \mathbb{R}^2, \ms\\
\ds u_{2\omega}^{(1)}|_{+}-u_{2\omega}^{(1)}|_{-} = I_3 \quad &\mbox{on} \ \partial \Omega,\ms \\
\ds \frac{\partial u_{2\omega}^{(1)}}{\partial \nu} \Big|_{+} - \varepsilon_{2\omega}\frac{\partial u_{2\omega}^{(1)}}{\partial \nu} \Big|_{-} = I_4 \quad &\mbox{on } \partial \Omega, \ms \\
\ds  u_{2\omega}^{(1)} = o(1)\ \ &as\ |x|\rightarrow +\infty,
\end{cases}
\end{align}
with
\begin{align}\label{eq:boundary-term}
\begin{cases}
\ds I_1 = f\left(\frac{\partial u_\omega^{(0)}}{\partial\nu}\Big|_{-}-\frac{\partial u_\omega^{(0)}}{\partial\nu}\Big|_{+}\right) \quad &\mbox{on } \partial \Omega,\ms \\
\ds I_2 = f \left(\varepsilon_\omega \frac{\partial^2 u_\omega^{(0)}}{\partial \nu^2}\Big|_{-} - \frac{\partial^2 u_\omega^{(0)}}{\partial \nu^2}\Big|_{+} \right) +  f'\left(\frac{\partial u_\omega^{(0)}}{\partial T}\Big|_{+} - \varepsilon_\omega\frac{\partial u_\omega^{(0)}}{\partial T}\Big|_{-} \right) \quad &\mbox{on } \partial \Omega, \ms \\
\ds I_3 = f\left(\frac{\partial u_{2\omega}^{(0)}}{\partial\nu}\Big|_{-}-\frac{\partial u_{2\omega}^{(0)}}{\partial\nu}\Big|_{+}\right) + 4\pi P^{s,(1)}_{\perp} \quad &\mbox{on} \ \partial \Omega,\ms \\
\ds I_4 = f \left(\varepsilon_{2\omega} \frac{\partial^2 u_{2\omega}^{(0)}}{\partial \nu^2}\Big|_{-} - \frac{\partial^2 u_{2\omega}^{(0)}}{\partial \nu^2}\Big|_{+} \right) +  f'\left(\frac{\partial u_{2\omega}^{(0)}}{\partial T}\Big|_{+} - \varepsilon_{2\omega}\frac{\partial u_{2\omega}^{(0)}}{\partial T}\Big|_{-} \right)  -4\pi \sigma^{s,(1)}  \quad &\mbox{on } \partial \Omega.
\end{cases}
\end{align}
\end{thm}

\subsection{Representation formulas}\label{sec:Representation-formula}
In this subsection, we establish the representation formulas for the solutions to the leading-order and first-order coupled systems \eqref{eq:leading-term-governing} and \eqref{eq:first-order-equation} in Theorem \ref{thm-expansion}.

Similar to Theorem \ref{well-posedness}, the solution $u_\omega^{(0)}$ to system \eqref{eq:leading-term-governing} can be expressed as
\begin{equation}\label{eq:u-omega-leading}
u_\omega^{(0)} = H(x) + \mathcal{S}_{\Omega}[\varphi_0](x),\quad x\in\mathbb{R}^2,
\end{equation}
where the density function $\varphi_0 \in L_0^2(\p \Omega)$ satisfies
\begin{align*}
%\label{eq:u-omega-leading-density}
(\lambda_{\omega} \mathcal{I} -\mathcal{K}^*_{\Omega})[\varphi_0] = \frac{\partial H}{\partial \nu} \quad \mbox{on } \p \Omega,
\end{align*}
with $\lambda_{\omega}$ given by \eqref{eq:lambda-omega}.

Furthermore, $u^{(0)}_{2\omega}$ can be represented using the double-layer potential $\mathcal{D}_{\Omega}$ and the single-layer potential $\mathcal{S}_{\Omega}$ as follows:
\begin{equation}\label{eq:u2-omega-leading}
u^{(0)}_{2\omega} =  \mathcal{D}_{\Omega}[\phi_0](x) + \mathcal{S}_{\Omega}[\psi_0](x),\quad x\in \mathbb{R}^2,
\end{equation}
where the pair $(\phi_0, \psi_0)\in L_0^2(\p \Omega)\times L_0^2(\p \Omega)$ satisfies
\begin{align}
\label{eq:u2-omega-leading-density}
\begin{cases}
\phi_0 = -4\pi P^{s,(0)}_{\perp} &\quad \mbox{on } \p \Omega, \\[4pt]
\ds (\lambda_{2\omega} \mathcal{I} -\mathcal{K}^*_{\Omega})[\psi_0] - \frac{\partial \mathcal{D}_{\Omega}}{\partial \nu}[\phi_0] = - 4\pi\frac{\sigma^{s,(0)} + \sigma^{b,(0)}}{\varepsilon_{2\omega}-1} &\quad \mbox{on } \p \Omega.
\end{cases}
\end{align}
Here $\lambda_{2\omega}$ is given by \eqref{eq:lambda-2omega}.

Moreover, the solution $u_{\omega}^{(1)}$ to \eqref{eq:first-order-equation} can be represented as
\begin{align}\label{eq:u-omega-first}
u_{\omega}^{(1)} = \mathcal{D}_{\Omega}[\varphi_1](x) + \mathcal{S}_{\Omega}[\varphi_2](x), \quad x\in\mathbb{R}^2,
\end{align}
where the pair $(\varphi_1, \varphi_2)\in L_0^2(\p \Omega)\times L_0^2(\p \Omega)$ satisfies
\begin{align}
\label{eq:u-omega-first-density}
\begin{cases}
\varphi_1 = -I_1 &\quad \mbox{on } \p \Omega, \\[4pt]
\ds (\lambda_{\omega} \mathcal{I} -\mathcal{K}^*_{\Omega})[\varphi_2] - \frac{\partial \mathcal{D}_{\Omega}}{\partial \nu}[\varphi_1] = I_2/(\varepsilon_{\omega}-1) &\quad \mbox{on } \p \Omega.
\end{cases}
\end{align}

Similarly, the solution $u_{2\omega}^{(1)}$ to \eqref{eq:first-order-equation} admits the representation
\begin{align}\label{eq:u2-omega-first}
u_{2\omega}^{(1)} =  \mathcal{D}_{\Omega}[\phi_1](x) + \mathcal{S}_{\Omega}[\psi_1](x), \quad x\in\mathbb{R}^2,
\end{align}
where the pair $(\phi_1, \psi_1)\in L_0^2(\p \Omega)\times L_0^2(\p \Omega)$ satisfies
\begin{align}
\label{eq:u2-omega-first-density}
\begin{cases}
\phi_1 = -I_3 &\quad \mbox{on } \p \Omega,\ms \\
\ds (\lambda_{2\omega} \mathcal{I} -\mathcal{K}^*_{\Omega})[\psi_1] - \frac{\partial \mathcal{D}_{\Omega}}{\partial \nu}[\phi_1] =  I_4/(\varepsilon_{2\omega}-1) &\quad \mbox{on } \p \Omega.
\end{cases}
\end{align}

\section{Nonlinear harmonic generation via geometric symmetry perturbations in a uniform background field}\label{sec:proof}
In this section, we consider the second harmonic generation under a uniform background field, i.e., $H(x) = -E x_1$, where $E$ is a positive constant and the domain $\Omega$ is a disk. We calculate the explicit form of the perturbation solution.
Throughout this section, we set $\Omega =\{|x|<r_0\}$. Then, from \eqref{eq:geometry}, $\p \Omega_\epsilon$ in polar coordinates is expressed as
\begin{align*}
\begin{cases}
\ds  x_{1}=(r_0+\epsilon f(\theta))\cos\theta,\ms\\
\ds  x_{2}=(r_0+\epsilon f(\theta))\sin\theta,
\end{cases}
\end{align*}
where $x =(x_{1}, x_{2})\in \p \Omega_\epsilon$.

One can easily see from \cite{Ammari2013-1} that for each positive integer $n$,
\begin{align}\label{eq:S-circle}
\mathcal{S}_\Omega [\e^{\i n \theta}] (x) =
\begin{cases}
\ds -\frac{r_0}{2n} \left( \frac{r}{r_0} \right)^{n} \e^{\i n\theta} & \text{if } |x| = r < r_0, \ms\\
\ds -\frac{r_0}{2n} \left( \frac{r_0}{r} \right)^{n} \e^{\i n\theta} & \text{if } |x| = r > r_0,
\end{cases}
\end{align}
and
\begin{align}\label{eq:K}
\mathcal{K}_{\Omega}^{*} [\e^{\i n\theta}] = 0 \quad \forall n \neq 0.
\end{align}

We also get
\begin{align}\label{eq:D-circle}
\mathcal{D}_\Omega [\e^{\i n\theta}] (x) =
\begin{cases}
\ds \frac{1}{2} \left( \frac{r}{r_0} \right)^{n} \e^{\i n\theta} & \text{if } |x| = r < r_0, \ms\\
\ds -\frac{1}{2} \left( \frac{r_0}{r} \right)^{n} \e^{\i n\theta} & \text{if } |x| = r > r_0,
\end{cases}
\end{align}
and
\begin{equation}\label{eq:D-circle-1}
\mathcal{D}_\Omega[1](x) =
\begin{cases} 
\ds 1 & \text{if } |x| = r < r_0, \ms \\ 
\ds 0 & \text{if } |x| = r > r_0.
\end{cases}
\end{equation}

We are ready to present the definitions of symmetry for the geometry of the material boundary.
\begin{defn}\label{defn:symmetry-breaking}
The two-dimensional geometric shape is called geometric symmetry breaking if it lacks inversion symmetry.
\end{defn}

\begin{rem}
Inversion symmetry in two-dimensional space refers to the transformation that maps any point 
$(x_1,x_2)$ to $(-x_1,-x_2)$, also known as central symmetry. This operation is equivalent to a \(180^\circ\) rotation about the origin and therefore does not alter chirality. A geometry possesses inversion symmetry if it remains unchanged under this transformation, such as a circle or a regular polygon with an even number of sides centered at the origin. 
\end{rem}

Although $f$ can be an arbitrary smooth function, for analytical convenience, we set $f =r_0 \cos(n\theta)$ in this paper. Then the boundary defined by \eqref{eq:geometry} can be rewritten as $r_s(\theta) = r_0(1+ \epsilon\cos(n \theta))$ in polar coordinates, as shown in Figure \ref{fig:perturbation-geometry}. It is worth noting that the boundary for $n=1$ would still remain circular with its center shifted by a distance equal to $\epsilon$ (see Figure \ref{fig:perturbation-geometry} (a)). As the original boundary’s geometry remains almost unchanged in this case, we therefore omit the case $n=1$ from consideration in the present study. Finally, according to Definition \ref{defn:symmetry-breaking}, it is evident that when $n\geq3$ is odd, the boundary exhibits symmetry breaking. Moreover, it also possesses rich symmetry properties, including rotations and reflections. To characterize these symmetries, we introduce the following symmetry degree.

\begin{figure}
  \centering
\begin{tikzpicture}
\foreach \angle in {0,90,180} {
    \draw[white, dashed, thin] (\angle:1.5) -- (\angle+180:1.5);
  }
\draw[red, domain=0:360, samples=200, smooth] 
   plot ({\x}: {1});
  \draw[domain=0:360, samples=200, smooth] 
    plot ({\x}: {1 + 0.2*cos(1*\x)});
  \node at (0, -2){$\text{(a) } n=1$};
\end{tikzpicture} 
\qquad 
\begin{tikzpicture}
  \draw[red, domain=0:360, samples=200, smooth] 
   plot ({\x}: {1});
\foreach \angle in {0,90,180} {
    \draw[blue, dashed, thin] (\angle:1.5) -- (\angle+180:1.5);
  }
  \draw[domain=0:360, samples=200, smooth] 
    plot ({\x}: {1 + 0.2*cos(2*\x)});
  \node at (0, -2){$\text{(b) } n=2$};
\end{tikzpicture} 
\qquad 
\begin{tikzpicture}
  \draw[red, domain=0:360, samples=200, smooth] 
   plot ({\x}: {1});
    \foreach \angle in {0,60,120} {
\draw[blue, dashed, thin] (\angle:1.5) -- (\angle+180:1.5);
  }
  \draw[domain=0:360, samples=200, smooth] 
    plot ({\x}: {1 + 0.2*cos(3*\x)});
  \node at (0, -2) {$\text{(c) } n=3$};
\end{tikzpicture}\\
\begin{tikzpicture}
\foreach \angle in {0,45,90,135} {
    \draw[blue, dashed, thin] (\angle:1.5) -- (\angle+180:1.5);
  }
  \draw[domain=0:360, samples=200, smooth] 
    plot ({\x}: {1 + 0.2*cos(4*\x)});
     \draw[red, domain=0:360, samples=200, smooth] 
   plot ({\x}: {1});
  \node at (0, -2) {$\text{(d) } n=4$};
\end{tikzpicture}
\qquad 
\begin{tikzpicture}
\foreach \angle in {0,36,72,108,144} {
    \draw[blue, dashed, thin] (\angle:1.5) -- (\angle+180:1.5);
  }
     \draw[red, domain=0:360, samples=200, smooth] 
   plot ({\x}: {1});
  \draw[domain=0:360, samples=200, smooth] 
    plot ({\x}: {1 + 0.2*cos(5*\x)});
  \node at (0, -2) {$\text{(e) } n=5$};
\end{tikzpicture}
\qquad 
\begin{tikzpicture}
\foreach \angle in {0,30,60,90,120,150} {
    \draw[blue, dashed, thin] (\angle:1.5) -- (\angle+180:1.5);
  }
   \draw[red, domain=0:360, samples=200, smooth] 
   plot ({\x}: {1});
  \draw[domain=0:360, samples=200, smooth] 
    plot ({\x}: {1 + 0.2*cos(6*\x)});
  \node at (0, -2){$\text{(f) } n=6$};
\end{tikzpicture}
\caption{Schematic of perturbation geometry (black solid line). The red solid and blue dashed lines correspond to a perfect circle and the reflection line, respectively. The reflection line for $n=1$ is omitted due to its infinite number.}\label{fig:perturbation-geometry}
\end{figure}
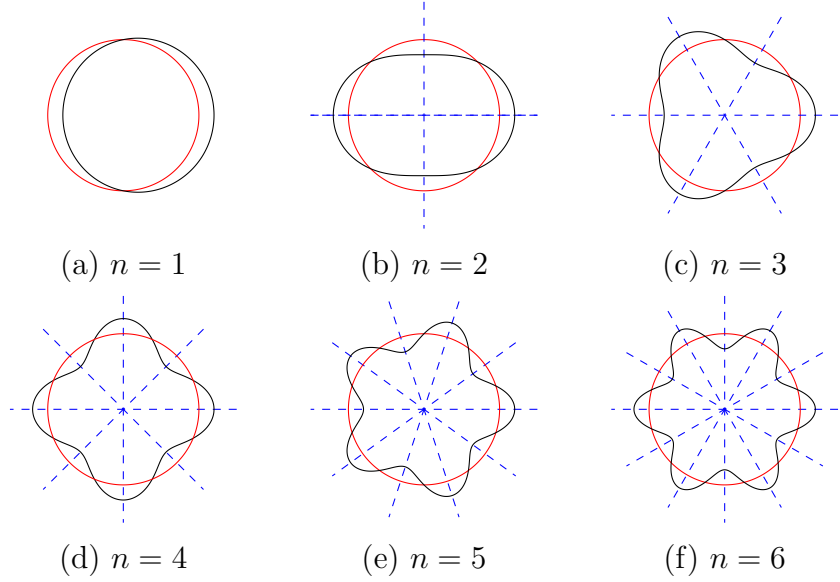

\begin{defn}\label{defn:symmetry-degree}
For a two-dimensional geometric shape, if there exists a point such that all rotations about that point and reflections across lines passing through that point map the shape onto itself, then the number of all such transformations (including the identity transformation) is called the symmetry degree of the shape. If there are infinitely many such transformations, the shape is said to have infinite symmetry degree.
\end{defn}

In fact, for the geometry studied in this paper, the symmetry degree defined in Definition \ref{defn:symmetry-degree} equals the order of the dihedral group $D_n$. Based on the properties of dihedral groups in Section \ref{sec:group-theory}, we know that when $n=3$, the geometry exhibits symmetry breaking and possesses the lowest degree of symmetry.

We will present the main results on how the symmetry of the geometry influences dipole radiation and multipole radiation.

%We are ready to present the proof of Theorem \ref{thm:mainthm1}.
\begin{thm}\label{thm:mainthm1}
Let the boundary of $\Omega_\epsilon$ be defined by (\ref{eq:geometry}) with $f = r_0 \cos(n\theta)$, and assume that the fundamental and second-harmonic fields satisfy the governing system (\ref{eq:governing-equation}) with $\Omega_\epsilon$. If the domain $\Omega_\epsilon$ is a disk, i.e., $\Omega_\epsilon = \Omega $ or $\epsilon=0$, then dipole radiation is forbidden.
\end{thm}

\begin{proof}
%Let $H(x) = - E r \cos\theta$. 
From the layer potential theory in Subsection \ref{sec:Representation-formula}, we have
\begin{equation*}
u_\omega^{(0)} =
H(x) + \mathcal{S}_{\Omega}[\varphi_0](x),\quad x\in\mathbb{R}^2,
\end{equation*}
where
\begin{align*}
\left(\lambda_{\omega} \mathcal{I} - \mathcal{K}^*_{\Omega}\right)[\varphi_0]= \frac{\partial H}{\partial r}\Big |_{r=r_0}.
\end{align*}
By \eqref{eq:K} and simple calculation, we obtain
\begin{align}\label{eq:varphi-circle}
  \varphi_0 = -\lambda_\omega^{-1} E \cos \theta.
\end{align}
Substituting \eqref{eq:varphi-circle} into \eqref{eq:u-omega-leading} and using \eqref{eq:S-circle}, we obtain the leading-order solution to \eqref{eq:leading-term-governing}
\begin{align}\label{eq:u-omega-leading-term}
 u_\omega^{(0)} =
  \begin{cases}
\ds -E\frac{2}{1+\varepsilon_\omega} r \cos\theta, &\quad r<r_0,
\vspace{1em}\\
\ds -E\left(r + \frac{1-\varepsilon_\omega}{1+\varepsilon_\omega} \frac{r_0^2}{r}\right) \cos\theta, &\quad r>r_0.
\end{cases}
 \end{align}

We next compute the leading-order solution for the second harmonic field. To this end, we first calculate $P_{\perp}^{s,(0)}$, $F^{(0)}$, $\sigma^{s,(0)}$ and $\sigma^{b,(0)}$. By straightforward calculation, we have
\begin{align}\label{eq:P-s_perp-lead}
P^{s,(0)}_{\perp}
= \chi_{\perp}^s \left(\varepsilon_\omega\frac{\p u^{(0)}_\omega}{\p \nu}\Big|_{-}\right)^2 
= 2\chi_{\perp}^s E^2  \left(\frac{\varepsilon_\omega}{1+\varepsilon_\omega}\right)^2  (1+\cos(2\theta)),
\end{align}
and 
\begin{align}\label{eq:sigma-s_perp-lead}
\sigma^{s,(0)} =  -\frac{1}{r_0} \frac{\d P^{s,(0)}_{\parallel}}{\d \theta}
= 8 \chi^s_{\parallel} E^2  \frac{\varepsilon_\omega}{r_0(1+\varepsilon_\omega)^2}  \cos(2\theta) .
\end{align}
Hence, solving equation \eqref{eq:u2-omega-leading-density} with \eqref{eq:P-s_perp-lead} and \eqref{eq:sigma-s_perp-lead} yields
\begin{align}\label{eq:u2-omega-leading-density-circle}
\begin{cases}
\ds  \phi_0 =  -8\pi \chi_{\perp}^s E^2  \Big(\frac{\varepsilon_\omega}{1+\varepsilon_\omega}\Big)^2  (1+\cos(2\theta)), \ms \\
\ds  \psi_0 =  - 8\pi\lambda_{2\omega}^{-1} E^2 \frac{1}{r_0(1+\varepsilon_\omega)^2} \Big(\chi_{\perp}^s \varepsilon_{\omega}^2 +\frac{4\chi^s_{\parallel} \varepsilon_\omega }{\varepsilon_{2\omega}-1} \Big) \cos(2\theta).  
\end{cases}  
\end{align}
Substituting \eqref{eq:u2-omega-leading-density-circle} into the integral expression \eqref{eq:u2-omega-leading} and using \eqref{eq:S-circle}, \eqref{eq:D-circle} and \eqref{eq:D-circle-1}, we obtain the leading-order solution for the second harmonic field
\begin{align}\label{eq:u2-omega-leading-circle}
  u_{2\omega}^{(0)} =\begin{cases}
                      \ds  - \chi_{\perp}^s\varepsilon_\omega^2    \frac{8\pi E^2}{(1+\varepsilon_\omega)^2} - \left(\chi_{\perp}^s \varepsilon_{\omega}^2 -2\chi^s_{\parallel} \varepsilon_{\omega}\right) \frac{ 8\pi E^2 }{(1+\varepsilon_\omega)^2(1+\varepsilon_{2\omega})}\frac{r^{2}}{r_0^2}\cos(2 \theta), & r<r_0, \ms \\[8pt]
                      \ds  \left(\chi_{\perp}^s \varepsilon_{\omega}^2 \varepsilon_{2\omega} +2\chi^s_{\parallel} \varepsilon_{\omega}\right)  \frac{8\pi E^2}{(1+\varepsilon_\omega)^2(1+\varepsilon_{2\omega})}\frac{r_0^2}{r^{2}}\cos(2 \theta), & r>r_0.
                     \end{cases}
\end{align}

From \eqref{eq:u2-omega-leading-circle}, it follows that for a perfect cylindrical nanowire, the leading-order solution for the second harmonic on the outside is solely the quadrupole term, decaying as \(O\big(1/r^2\big)\). Hence, dipole radiation is forbidden.

The proof is complete.
\end{proof}

%We next give the proof of Theorem \ref{thm:mainthm2}.

\begin{thm}\label{thm:mainthm2}
Under the same assumptions as in Theorem \ref{thm:mainthm1}, dipole radiation exists if the boundary $\p \Omega_\epsilon$ exhibits geometric symmetry breaking and possesses the lowest degree of symmetry, that is, $f = r_0 \cos(3\theta)$. If $\p \Omega_\epsilon$ possesses a $2n$-order degree of symmetry, i.e., $f(\theta)= r_0 \cos(n\theta)$, then the second-harmonic field is $2^{n-2}$-multipole radiation.
\end{thm}

\begin{proof}
For simplicity, we first write $f$ as 
\begin{align}\label{f-series}
f(\theta) = r_0\cos(n\theta), \quad n\geq 3,
\end{align}
where $n$ is a positive integer. 
Using \eqref{eq:boundary-term}, \eqref{eq:u-omega-leading-term} and \eqref{f-series}, we obtain  
\begin{align}\label{eq:I1}
  I_1 &= -2E \frac{1-\varepsilon_\omega}{1+\varepsilon_\omega}  f(\theta) \cos \theta \nonumber\\
   &= -E\frac{1-\varepsilon_\omega}{1+\varepsilon_\omega} r_0 [\cos((n-1)\theta)+\cos((n+1)\theta)],
\end{align}
and
\begin{align}\label{eq:I2}
  I_2 &= 2E \frac{1-\varepsilon_\omega}{r_0(1+\varepsilon_\omega)} \left(f(\theta) \cos \theta +  \frac{\d f(\theta)}{\d \theta}\sin \theta \right)\nonumber\\
  &= - E\frac{1-\varepsilon_\omega}{1+\varepsilon_\omega} [(n-1)\cos((n-1)\theta)-(n+ 1)\cos((n+1)\theta) ].
\end{align}
Solving equation \eqref{eq:u-omega-first-density} with \eqref{eq:I1} and \eqref{eq:I2} yields
\begin{align}\label{eq:varphi-12}
  \begin{cases}
\ds \varphi_1 = E\frac{1-\varepsilon_\omega}{1+\varepsilon_\omega} r_0 [\cos((n-1)\theta)+\cos((n+1)\theta)], \ms\\
\ds \varphi_2 = \lambda_\omega^{-1} \frac{E}{2} \left(\frac{3-\varepsilon_\omega}{1+\varepsilon_\omega}(n-1)\cos((n-1)\theta)- (n+1)\cos((n+1)\theta)\right) .
  \end{cases}
\end{align}
Substituting \eqref{eq:varphi-12} into \eqref{eq:u-omega-first} and using \eqref{eq:S-circle} and \eqref{eq:D-circle}, we obtain the first-order solution for the linear field
\begin{align*}
 u_\omega^{(1)} =
  \begin{cases}
\ds 2 E\frac{1-\varepsilon_\omega}{(1+\varepsilon_\omega)^2} \frac{r^{n-1}}{r_0^{n-2}} \cos((n-1)\theta), & r<r_0, \\[8pt]
\ds E \left(\left(\frac{1-\varepsilon_\omega}{1+\varepsilon_\omega}\right)^2 \frac{r_0^n}{r^{n-1}} \cos((n-1)\theta)- \frac{1-\varepsilon_\omega}{1+\varepsilon_\omega} \frac{r_0^{n+2}}{r^{n+1}} \cos((n+1)\theta)\right), & r>r_0.
\end{cases}
 \end{align*}

We now compute the first-order solution for the second harmonic field. To compute $I_3$, we need to derive the following equations:
\begin{align*}
f \left(\frac{\partial u_{2\omega}^{(0)}}{\partial \nu}\Big|_{-}- \frac{\partial u_{2\omega}^{(0)}}{\partial \nu}\Big|_{+} \right)
= \left(\chi_{\perp}^s \varepsilon_{\omega}^2 (\varepsilon_{2\omega}-1) + 4\chi^s_{\parallel} \varepsilon_{\omega}\right) \frac{ 8\pi E^2 }{(1+\varepsilon_\omega)^2(1+\varepsilon_{2\omega})} [\cos((n-2)\theta) + \cos((n+2)\theta)]
\end{align*}
and
\begin{align*}
P^{s,(1)}_{\perp} 
&= -2 \chi_{\perp}^s E^2 \left(\frac{\varepsilon_\omega}{1+\varepsilon_\omega}\right)^2 \left(
 2\frac{1-\varepsilon_\omega}{1+\varepsilon_\omega} (n-1) \cos((n-2)\theta) + n  \cos((n-2)\theta) \right)\\
 &\quad - 2 \chi_{\perp}^s E^2 \left(\frac{\varepsilon_\omega}{1+\varepsilon_\omega}\right)^2 \left( 2\frac{1-\varepsilon_\omega}{1+\varepsilon_\omega} (n-1)\cos(n\theta) - n \cos((n+2)\theta) \right).   
\end{align*} 
Combining these expressions, we obtain 
\begin{align}\label{eq:I3-circle}
I_3 &= \frac{ 8\pi E^2}{(1+\varepsilon_\omega)^2}\left[\frac{\varepsilon_{\omega}}{1+\varepsilon_{2\omega}} \left(\chi_{\perp}^s \varepsilon_{\omega} (\varepsilon_{2\omega}-1) + 4\chi^s_{\parallel}\right) - \left(2\frac{1-\varepsilon_\omega}{1+\varepsilon_\omega} (n-1) + n\right) \chi_{\perp}^s \varepsilon_{\omega}^2 \right] \cos((n-2)\theta) \nonumber\\
 &\quad -\frac{ 8\pi E^2 }{(1+\varepsilon_\omega)^2}\left(2\frac{1-\varepsilon_\omega}{1+\varepsilon_\omega} (n-1) \chi_{\perp}^s \varepsilon_{\omega}^2  \right) \cos(n\theta) \nonumber \\
 &\quad +\frac{ 8\pi E^2 }{(1+\varepsilon_\omega)^2}\left[ \frac{\varepsilon_{\omega}}{1+\varepsilon_{2\omega}} \left(\chi_{\perp}^s \varepsilon_{\omega} (\varepsilon_{2\omega}-1) + 4\chi^s_{\parallel}\right) + n \chi_{\perp}^s \varepsilon_{\omega}^2  \right] \cos((n+2)\theta).
\end{align}

We next derive the following expressions needed to compute $I_4$. The terms involving $f$ and $f'$ yield
\begin{align*}
&f \left(\varepsilon_{2\omega} \frac{\partial^2 u_{2\omega}^{(0)}}{\partial \nu^2}\Big|_{-}- \frac{\partial^2 u_{2\omega}^{(0)}}{\partial \nu^2}\Big|_{+} \right)\\
=& \left(-2\chi_{\perp}^s \varepsilon_{\omega}^2 \varepsilon_{2\omega} + \chi^s_{\parallel} \varepsilon_{\omega} \varepsilon_{2\omega} - 3\chi^s_{\parallel}\varepsilon_{\omega} \right) \frac{ 16\pi E^2 }{r_0(1+\varepsilon_\omega)^2(1+\varepsilon_{2\omega})} [ \cos((n-2)\theta) + \cos((n+2)\theta)],
\end{align*}
and
\begin{align*}
&f'\left(\frac{\partial u_{2\omega}^{(0)}}{\partial T}\Big|_{+} - \varepsilon_{2\omega}\frac{\partial u_{2\omega}^{(0)}}{\partial T}\Big|_{-} \right) \\
=& \left(\chi_{\perp}^s \varepsilon_{\omega}^2 \varepsilon_{2\omega} - \chi^s_{\parallel} \varepsilon_{\omega} \varepsilon_{2\omega} + \chi^s_{\parallel}\varepsilon_{\omega} \right) \frac{ 16\pi E^2}{r_0 (1+\varepsilon_\omega)^2(1+\varepsilon_{2\omega})} [ n \cos((n-2)\theta) -n \cos((n+2)\theta) ].
\end{align*}
Additionally, the expansion of the surface and bulk nonlinear polarizations defined by \eqref{eq:sigma-s-expansion} contribute
\begin{align*}
  \sigma^{s,(1)}
  &=-4 \chi_{\parallel}^s E^2 \frac{\varepsilon_\omega}{r_0 (1+\varepsilon_\omega)^2}  \Big(\cos((n-2)\theta)+\cos((n+2)\theta)+2\Big( \frac{1-\varepsilon_\omega}{1+\varepsilon_\omega}\Big) (n-1)n \cos(n\theta) \\
  &\quad -  n(n-2) \cos((n-2)\theta) - n(n+2) \cos((n+2)\theta) \Big).
\end{align*}
From these expressions, it follows that 
\begin{align}\label{eq:I4-circle}
  I_4 
 &= \frac{ 16\pi E^2 }{r_0(1+\varepsilon_\omega)^2}\left[\frac{\varepsilon_{\omega}}{1+\varepsilon_{2\omega}} \Big( \chi_{\perp}^s \varepsilon_{\omega} \varepsilon_{2\omega} - \chi^s_{\parallel}\varepsilon_{2\omega} + \chi^s_{\parallel}\Big) - n\chi_{\parallel}^s \varepsilon_{\omega}  \right] (n-2)\cos((n-2)\theta) \nonumber\\
 &\quad +\frac{ 16\pi E^2 }{r_0(1+\varepsilon_\omega)^2}\left(2\frac{1-\varepsilon_\omega}{1+\varepsilon_\omega} (n-1) \chi_{\parallel}^s \varepsilon_{\omega} \right) n \cos(n\theta) \nonumber\\
 &\quad -\frac{ 16\pi E^2 }{r_0(1+\varepsilon_\omega)^2}\left[\frac{\varepsilon_{\omega}}{1+\varepsilon_{2\omega}} \Big( \chi_{\perp}^s \varepsilon_{\omega} \varepsilon_{2\omega} - \chi^s_{\parallel}\varepsilon_{2\omega} + \chi^s_{\parallel}\Big) + n\chi_{\parallel}^s\varepsilon_{\omega} \right](n+2)\cos((n+2)\theta).
\end{align}
 
If $\phi_1$ and $\psi_1 $ are given by 
\begin{align*}
\left[
  \begin{array}{ccc}
    \phi_{1} \ms \\
    \psi_{1}
  \end{array}
\right]
=
\sum_{m\in\{-2,0,2\}}
\left[
  \begin{array}{cc}
    \phi_{1,n+p} \ms \\
    \phi_{1,n+p}
  \end{array}
\right] \cos((n+p)\theta),
\end{align*}
from \eqref{eq:u2-omega-first-density} we obtain
\begin{align}\label{eq:u2-omega-first-density-circle}
\begin{cases}
\ds \phi_1 = -I_3, \ms \\[4pt]
\ds \psi_1 = \lambda_{2\omega}^{-1} \left(\frac{I_4}{\varepsilon_{2\omega}-1} + \sum_{p\in\{-2,0,2\}} \frac{n+p}{2r_0} \phi_{1,n+p}\cos((n+p)\theta) \right).
\end{cases} 
\end{align}
Substituting \eqref{eq:I3-circle}, \eqref{eq:I4-circle} and \eqref{eq:u2-omega-first-density-circle} into the integral expression \eqref{eq:u2-omega-first} and using \eqref{eq:S-circle} and \eqref{eq:D-circle}, we obtain the first-order solution for the second harmonic field:
\begin{align}\label{eq:u2omega-1}
  u_{2\omega}^{(1)} = M_{n-2} \frac{1}{r^{n-2}}\cos((n-2)\theta) + M_{n} \frac{1}{r^{n}}\cos(n\theta) + M_{n+2} \frac{1}{r^{n+2}}\cos((n+2)\theta), \quad n\geq 3,
\end{align}
where
\begin{align}
M_{n-2} &=  \frac{8\pi E^2 }{(1+\varepsilon_{2\omega})(1+\varepsilon_\omega)^2}\bigg(\frac{\varepsilon_{2\omega}-3}{1+\varepsilon_{2\omega}}\chi_{\perp}^s\varepsilon^2_{\omega}\varepsilon_{2\omega} + \frac{(n-2)\varepsilon_\omega -(3n -2)}{1+\varepsilon_\omega}\chi_{\perp}^s\varepsilon^2_{\omega}\varepsilon_{2\omega} \nonumber  \\
  &\quad + 2\frac{(n-1)+(n+3)\varepsilon_{2\omega}}{1+\varepsilon_{2\omega}}\chi_{\parallel}^s \varepsilon_{\omega} \bigg)r_0^{n-2}, \label{eq:Mn-2}  \\
M_{n}&=\frac{ 16\pi E^2 }{(1+\varepsilon_{2\omega})(1+\varepsilon_\omega)^2}\left(\frac{\varepsilon_{\omega}-1}{1+\varepsilon_{\omega}}(n-1)\big(\chi_{\perp}^s\varepsilon^2_{\omega}\varepsilon_{2\omega}+2 \chi_{\parallel}^s \varepsilon_{\omega}\big)\right) r_0^{n}, \nonumber \\
M_{n+2} &= \frac{8\pi E^2 }{(1+\varepsilon_{2\omega})(1+\varepsilon_\omega)^2}(n+1) \left(\chi_{\perp}^s\varepsilon^2_{\omega}\varepsilon_{2\omega}+2 \chi_{\parallel}^s \varepsilon_{\omega}\right)r_0^{n+2} .\nonumber
\end{align}
From \eqref{eq:u2omega-1}, we observe that the dipole radiation occurs when $n=3$.

This completes the proof.
\end{proof}

\begin{thm}\label{thm:mainthm3}
Let the same assumptions as in Theorem \ref{thm:mainthm1} hold. The following three types of plasmon resonance occur in the second harmonic field:
\begin{itemize}
  \item If $\varepsilon_\omega = -1$ and $\varepsilon_{2\omega} \neq -1$, then the resonance order is $O\left((1+\varepsilon_\omega)^{-3}\right)$. 
  \item If $\varepsilon_\omega \neq -1$ and $\varepsilon_{2\omega} = -1$, then the resonance order is $O\left((1+\varepsilon_{2\omega})^{-2}\right)$. 
  \item If $\varepsilon_\omega = -1$ and $\varepsilon_{2\omega} = -1$, then the resonance order is either
   $O\left((1+\varepsilon_\omega)^{-2}(1+\varepsilon_{2\omega})^{-2}\right)$ or \\ $O\left((1+\varepsilon_\omega)^{-3}(1+\varepsilon_{2\omega})^{-1}\right)$.  
\end{itemize}
\end{thm}

\begin{proof}
 The proof of Theorem \ref{thm:mainthm3} is straightforward by combining equations \eqref{eq:u2-omega-leading-circle}, \eqref{eq:u2omega-1} and \eqref{eq:Mn-2}.
\end{proof}

\begin{rem}
From Theorem \ref{thm:mainthm2} and Theorem \ref{thm:mainthm3}, it can be seen that under plasmonic resonance and symmetry breaking, not only is the intensity of the second harmonic field enhanced, but its propagation distance is also increased. Moreover, the resonance enhancement of second harmonic generation is strongest when resonance occurs simultaneously in the linear and second harmonic electric fields, noting that resonance occurs in the linear electric field when $\varepsilon_\omega =-1$. Conversely, if the nonlinear-material boundary possesses a higher-order degree of symmetry (i.e., it remains invariant under a higher-order dihedral group), then the second-harmonic field may include higher-order multipole radiation modes. However, these modes are characterized by a reduced radiative intensity, a faster decay rate, and, consequently, a shorter effective propagation distance. In practical applications, the limited propagation length of multipole radiation modes makes them difficult to detect experimentally.
\end{rem}

\section{Nonlinear harmonic generation via symmetry perturbations in the non-uniform background field}\label{sec:proof6}
In this section, we investigate the effects of inhomogeneous background fields on nonlinear optical responses. By introducing the concept of relative symmetry to quantify the degree of symmetry breaking in the background field and the geometry, we have demonstrated that dipole radiation occurs when the degree of symmetry breaking is minimal.

\subsection{Symmetry perturbations of non-uniform background field}
Assume the boundary is a circle. We first consider that the background field involves single homogeneous real harmonic term, that is, \[
H(x_1, x_2) =   C_\ell \cdot \operatorname{Re}(z^\ell)=  C_\ell r^\ell \cos(\ell\theta),
\]
where $z= x_1+ \i x_2$ is a complex variable. 
To preserve the physical meaning, we set $C_\ell  =- E$ such that $H(x) = - E r^\ell \cos(\ell\theta), \ell \geq 2$.  Following a calculation process similar to the proof of Theorem \ref{thm:mainthm1}, we have
\begin{align}\label{eq:u-omega-leading-term}
 u_\omega =
  \begin{cases}
\ds -E\frac{2}{1+\varepsilon_\omega} r^\ell \cos(\ell\theta), &\quad r<r_0,
\vspace{1em}\\
\ds -E\Big(r^\ell + \frac{1-\varepsilon_\omega}{1+\varepsilon_\omega} \frac{r_0^{2\ell}}{r^\ell}\Big) \cos(\ell\theta), &\quad r>r_0.
\end{cases}
 \end{align}
By straightforward calculation, it follows that
\begin{align*}
P^{s}_{\perp}
= \chi_{\perp}^s \left(\varepsilon_\omega\frac{\p u_\omega}{\p \nu}\Big|_{-}\right)^2 
= 2\chi_{\perp}^s E^2  \left(\frac{\varepsilon_\omega}{1+\varepsilon_\omega}\right)^2 \ell^2 r_0^{2(\ell-1)} (1+\cos(2\ell\theta))
\end{align*} 
and
\begin{align*}
\sigma^{s} &= -\frac{1}{r_0} \frac{\d P^{s}_{\parallel}}{\d \theta}
= 8 \chi^s_{\parallel} E^2  \frac{\varepsilon_\omega}{(1+\varepsilon_\omega)^2} \ell^3 r_0^{2\ell-3}  \cos(2 \ell \theta) .
\end{align*}
Hence, the second harmonic field is given by
\begin{align}\label{eq:u2-omega-Inhomogeneous}
  u_{2\omega} =\begin{cases}
                      \ds  - \chi_{\perp}^s \varepsilon_\omega^2  \frac{8\pi E^2 \ell^2}{(1+\varepsilon_\omega)^2}   r_0^{2(\ell-1)} - \left(\chi_{\perp}^s \varepsilon_{\omega}^2 -2\chi^s_{\parallel} \varepsilon_{\omega}\right) \frac{ 8\pi E^2 \ell^2}{(1+\varepsilon_\omega)^2(1+\varepsilon_{2\omega})}\frac{r^{2\ell}}{r_0^2}\cos(2 \ell \theta), & r<r_0, \ms \\
                      \ds  \left(\chi_{\perp}^s \varepsilon_{\omega}^2 \varepsilon_{2\omega} +2\chi^s_{\parallel}\varepsilon_{\omega}\right) \frac{ 8\pi E^2 \ell^2 }{(1+\varepsilon_\omega)^2(1+\varepsilon_{2\omega})}\frac{r_0^{4\ell-2}}{r^{2\ell}}\cos(2 \ell \theta), & r>r_0.
                     \end{cases}
\end{align}

From \eqref{eq:u2-omega-Inhomogeneous}, it can be seen that the dipole radiation is forbidden for the non-uniform background field $H(x) = - E r^\ell \cos(\ell\theta), \ell \geq 2$. To excite dipolar radiation passing through a non-uniform background field, we next construct the background field involving 
two homogeneous real harmonic terms, that is, \[
H(x_1, x_2) =  C_m \cdot \operatorname{Re}(z^m) + C_\ell \cdot \operatorname{Re}(z^\ell)= C_m r^m \cos(m \theta) + C_\ell r^\ell \cos(\ell\theta).
\]
Since two distinct non-uniform fields, or perturbations of the original non-uniform background field, are involved here, we introduce the following definition of relative symmetry degree.
\begin{defn}\label{resyde1}
The relative symmetry degree is defined as
\[
d_{\text{rel}} = |m - \ell| \geq 1, \quad (m \neq \ell).
\]
\end{defn}
According to Theorem \ref{thm:q-order}, we know that when the $\ell$-homogeneous real harmonic parts of \(H\) is invariant under the dihedral group \(D_q\), we have \(q \mid \ell\); hence \(q_{\text{max}} = \ell\). Thus, fixing \(\ell\), we discuss the relative symmetry degree of the freely varying \(m\) with respect to \(\ell\). Clearly, the larger \(m\) is, the more dihedral groups leave it invariant, hence the more symmetric it is. Therefore, the most asymmetric case is \(d_{\text{rel}} = 1\), i.e., \(m = \ell+1\) or \(\ell-1\). Therefore, we can obtain the following theorem.  
\begin{thm}
  Let $H = -E(r^m \cos(m \theta) + r^\ell \cos(\ell\theta))$ with $\ell \geq2$ and $m \geq1$. The relative symmetry degree satisfies $d_{\text{rel}} =1$, the dipole radiation occurs. Furthermore,  If $d_{\text{rel}} =n$, the $2^n$-multipole radiation occurs. Additionally, the resonant behavior of both dipole and multipole radiation is described as follows:
  \begin{itemize}
  \item If $\varepsilon_\omega = -1$ and $\varepsilon_{2\omega} \neq -1$, then the resonance order is $O\left((1+\varepsilon_\omega)^{-2}\right)$. 
  \item If $\varepsilon_\omega \neq -1$ and $\varepsilon_{2\omega} =-1$, then the resonance order is $O\left((1+\varepsilon_{2\omega})^{-1}\right)$. 
  \item If $\varepsilon_\omega =-1$ and $\varepsilon_{2\omega} =-1$, then the resonance order is $O\left((1+\varepsilon_\omega)^{-2}(1+\varepsilon_{2\omega})^{-1}\right)$.  
\end{itemize}
\end{thm}
\begin{proof}

 Given this background field, our calculations yield
\begin{align}\label{eq:u-omega-Inhomogeneous}
 u_\omega =
  \begin{cases}
\ds -E\frac{2}{1+\varepsilon_\omega} \Big( r^{m} \cos(m\theta) +r^\ell \cos(\ell\theta)  \Big), &\quad r<r_0,
\ms\\
\ds -E \left(\left(r^{m} + \frac{1-\varepsilon_\omega}{1+\varepsilon_\omega} \frac{r_0^{2m}}{r^{m}}\right) \cos(m\theta) +  \left( r^\ell + \frac{1-\varepsilon_\omega}{1+\varepsilon_\omega} \frac{r_0^{2\ell}}{r^\ell}\right) \cos(\ell\theta)  \right), &\quad r>r_0.
\end{cases}
\end{align}
By \eqref{eq:u-omega-Inhomogeneous}, we derive
\begin{align*}
P^s_{\perp} 
  &= 2\chi_{\perp}^s E^2  \left(\frac{\varepsilon_\omega}{1+\varepsilon_\omega}\right)^2 \Big(m^2 r_0^{2(m-1)} (1+\cos(2m\theta)) + \ell^2 r_0^{2(\ell-1)} (1+\cos(2\ell\theta)) \\
  &\quad + 2m\ell r_0^{m+\ell-2}(\cos((m-\ell)\theta) + \cos(m+\ell)\theta)\Big)
\end{align*}
and 
\begin{align*}
  \sigma^s 
= 8 \chi^s_{\parallel} E^2  \frac{\varepsilon_\omega}{(1+\varepsilon_\omega)^2} \left(m^3 r_0^{2m-3}  \cos(2 m \theta) + \ell^3 r_0^{2\ell-3}  \cos(2 \ell \theta) + m \ell (m+\ell) r_0^{m+\ell-3}  \cos((m+\ell) \theta) \right).
\end{align*}
Therefore, using \eqref{eq:sol-u2} and \eqref{eq:u2-density}, we can obtain
\begin{align}\label{eq:u2-omega-Inhomogeneous-v2}
  u_{2\omega} &=  \left(\chi_{\perp}^s \varepsilon_{\omega}^2 \varepsilon_{2\omega} +2\chi^s_{\parallel}\varepsilon_{\omega}\right) \frac{ 8\pi E^2 }{(1+\varepsilon_\omega)^2(1+\varepsilon_{2\omega})} \Big(m^2\frac{ r_0^{4m-2}}{r^{2m}}\cos(2 m \theta) + \ell^2\frac{r_0^{4\ell-2}}{r^{2\ell}}\cos(2 \ell \theta) \nonumber\\
    &\quad  + 2m\ell \frac{r_0^{2(m+\ell-1)}}{r^{m+\ell}}\cos((m+\ell)\theta)\Big)\nonumber \\
    &\quad + \chi_{\perp}^s \varepsilon_{\omega}^2 \varepsilon_{2\omega} \frac{ 16\pi E^2 m\ell }{(1+\varepsilon_\omega)^2(1+\varepsilon_{2\omega})}  \frac{r_0^{|m-\ell|+m+\ell-2}}{r^{|m-\ell|}}\cos((m-\ell)\theta), \quad r>r_0.
\end{align}
As can be seen from \eqref{eq:u2-omega-Inhomogeneous-v2}, dipole radiation occurs when $|m-\ell| = 1$. 

This completes the proof.
\end{proof}

\subsection{Symmetry perturbations of geometric perturbations in a non-uniform background field}\label{sec:Sb-non-uniform}
In this subsection, we consider perturbing the geometry to excite dipole radiation in a non-uniform background field $H(x) = - E r^\ell \cos(\ell\theta), \ell \geq 2$. Assume that \(r_s = r_0 + \epsilon f(\theta) \). Here
we still write $f$ as
\begin{align}\label{f-series}
f(\theta)=r_0\cos(n\theta), \quad n\geq 3.
\end{align}

Fix \(\ell\) and consider the relative position \(2\ell\), because \(D_\ell\) is an invariant subgroup of \(D_{2\ell}\), i.e., \(D_\ell\) is invariant under \(D_{2\ell}\): for every \(a \in D_{2\ell}\), we have \(a \cdot D_\ell = D_\ell \cdot a\). Then we can extend definition \ref{resyde1} for relative symmetry degree to the following definition.
\begin{defn}
The relative symmetry degree is defined as
\[
d_{\text{rel}} = |n - 2\ell| \geq 1, \quad n \neq 2\ell.
\]
\end{defn}

When \(n = 2\ell\), \(D_{\ell}\) is an invariant subgroup of \(D_{n}\) (trivial), so the minimal symmetry degree (most asymmetric) is also \(d_{\text{rel}} = 1\), i.e., \(n = 2\ell+1\) or \(2\ell-1\). Therefore, we can obtain the following theorem, which can be regarded as an extension of Theorem \ref{thm:mainthm2}. 
\begin{thm}
Let $H(x) = - E r^\ell \cos(\ell\theta), \ell \geq 2$ and $f(\theta)=r_0\cos(n\theta)$. The relative symmetry degree satisfies $d_{\text{rel}} =1$, the dipole radiation occurs. Furthermore,  If $d_{\text{rel}} =m$, the $2^m$-multipole radiation occurs. In addition, dipole and multipole radiation exhibit the following resonant behavior:
\begin{itemize}
  \item If $\varepsilon_\omega = -1$ and $\varepsilon_{2\omega} \neq -1$, then the resonance order is $O\left((1+\varepsilon_\omega)^{-3}\right)$. 
  \item If $\varepsilon_\omega \neq -1$ and $\varepsilon_{2\omega} =-1$, then the resonance order is $O\left((1+\varepsilon_{2\omega})^{-2}\right)$. 
  \item If $\varepsilon_\omega =-1$ and $\varepsilon_{2\omega} =-1$, then the resonance order is either $O\left((1+\varepsilon_\omega)^{-2}(1+\varepsilon_{2\omega})^{-2}\right)$ or \\ $O\left((1+\varepsilon_\omega)^{-3}(1+\varepsilon_{2\omega})^{-1}\right)$.  
\end{itemize}
\end{thm}
\begin{proof}
Note that the leading-order solution for the linear field has already been given by \eqref{eq:u-omega-leading-term}.
Using \eqref{eq:boundary-term} and \eqref{eq:u-omega-leading-term}, we can obtain  
\begin{align}\label{eq:I1}
  I_1 =& -2E \frac{1-\varepsilon_\omega}{1+\varepsilon_\omega} \ell r_0^{\ell -1} f(\theta) \cos(\ell \theta)\nonumber\\
   =& -E\frac{1-\varepsilon_\omega}{1+\varepsilon_\omega} \ell r_0^{\ell} [\cos((n-\ell)\theta)+\cos((n+\ell)\theta)]
\end{align}
and
\begin{align}\label{eq:I2}
  I_2 =& 2E \frac{1-\varepsilon_\omega}{1+\varepsilon_\omega} \ell r_0^{\ell -2} \left(\ell f(\theta) \cos(\ell \theta) + \frac{\d f(\theta)}{\d \theta}\sin(\ell \theta) \right)\nonumber\\
  =& - E\frac{1-\varepsilon_\omega}{1+\varepsilon_\omega} \ell r_0^{\ell -1} [(n-\ell)\cos((n-\ell)\theta)-(n + \ell)\cos((n +\ell)\theta)].
\end{align}
 
By \eqref{eq:u-omega-first} and \eqref{eq:u-omega-first-density}, we can have first-order solution for linear field  
\begin{align*}
%\label{annulus-p}
 u_\omega^{(1)} =
  \begin{cases}
\ds  E \ell \frac{1-\varepsilon_\omega}{(1+\varepsilon_\omega)^2}  \left(\frac{n-\ell}{|n-\ell|}+1\right)\frac{r^{|n-\ell|}}{r_0^{|n-\ell|-\ell}}\cos((n-\ell)\theta), & r<r_0,
\ms\\
\ds E \ell  \left(\frac{1-\varepsilon_\omega}{(1+\varepsilon_\omega)^2}  \left(\frac{n-\ell}{|n-\ell|}- \varepsilon_\omega\right) \frac{r_0^{|n-\ell|+\ell}}{r^{|n-\ell|}} \cos((n-\ell)\theta)- \frac{1-\varepsilon_\omega}{1+\varepsilon_\omega} \frac{r_0^{n+2\ell}}{r^{n+\ell}} \cos((n+\ell)\theta)\right), & r>r_0.
\end{cases}
 \end{align*} 
It is worth noting that when $n<\ell$, $u_\omega^{(1)} = 0$ for $r<r_0$ and when $n=\ell$, $u_\omega^{(1)} = \text{constant}$ for $r<r_0$. These conditions prevent the occurrence of second harmonics. Therefore, we mainly consider the case $n>\ell$.
Furthermore, for $n>\ell$, it follows that
 \begin{align*}
 %\label{annulus-p}
 u_\omega^{(1)} =
  \begin{cases}
\ds 2 E \ell \frac{1-\varepsilon_\omega}{(1+\varepsilon_\omega)^2} \frac{r^{n-\ell}}{r_0^{n-2\ell}}\cos((n-\ell)\theta), & r<r_0,
\vspace{1em}\\
\ds E \ell \left(\left(\frac{1-\varepsilon_\omega}{1+\varepsilon_\omega}\right)^2   \frac{r_0^{n}}{r^{n-\ell}} \cos((n-\ell)\theta)- \frac{1-\varepsilon_\omega}{1+\varepsilon_\omega} \frac{r_0^{n+2\ell}}{r^{n+\ell}} \cos((n+\ell)\theta)\right) , & r>r_0.
\end{cases}
 \end{align*}
 
We also note that the leading-order solution for the second harmonic field has already been given by \eqref{eq:u2-omega-Inhomogeneous}. Based on this, we proceed to calculate the first-order solution of the second harmonic.
Following a calculation process similar to the proof of Theorem \ref{thm:mainthm2}, we derive
\begin{align*}
I_3 =& \frac{ 8\pi E^2 \ell^2 r_0^{2\ell-2}}{(1+\varepsilon_\omega)^2}\left[ \frac{\varepsilon_{\omega}}{1+\varepsilon_{2\omega}} \ell \left(\chi_{\perp}^s \varepsilon_{\omega} (\varepsilon_{2\omega}-1) + 4\chi^s_{\parallel}\right) - \left(2\frac{1-\varepsilon_\omega}{1+\varepsilon_\omega} (n-\ell) + (n-\ell+1)\right) \chi_{\perp}^s \varepsilon_{\omega}^2 \right] \cos((n-2\ell)\theta)\\
 &-\frac{ 8\pi E^2 \ell^2 r_0^{2\ell-2}}{(1+\varepsilon_\omega)^2}\left[\left(2\frac{1-\varepsilon_\omega}{1+\varepsilon_\omega} (n-\ell) - 2(\ell-1)\right) \chi_{\perp}^s \varepsilon_{\omega}^2  \right] \cos(n\theta)\\
 &+\frac{ 8\pi E^2 \ell^2 r_0^{2\ell-2}}{(1+\varepsilon_\omega)^2}\left[ \frac{\varepsilon_{\omega}}{1+\varepsilon_{2\omega}} \ell  \left(\chi_{\perp}^s \varepsilon_{\omega}(\varepsilon_{2\omega}-1) + 4\chi^s_{\parallel} \right) + (n+\ell-1) \chi_{\perp}^s \varepsilon_{\omega}^2  \right] \cos((n+2\ell)\theta)
\end{align*}
and
\begin{align*}
  I_4 
 =& \frac{ 16\pi E^2 \ell^2 r_0^{2\ell-3}}{(1+\varepsilon_\omega)^2}\left[\frac{\varepsilon_{\omega}}{1+\varepsilon_{2\omega}} \ell \left( \chi_{\perp}^s \varepsilon_{\omega} \varepsilon_{2\omega} - \chi^s_{\parallel}\varepsilon_{2\omega} + \chi^s_{\parallel}\right) - (n-\ell+1)\chi_{\parallel}^s \varepsilon_{\omega} \right] (n-2\ell)\cos((n-2\ell)\theta)\\
 &+\frac{ 16\pi E^2 \ell^2 r_0^{2\ell-3}}{(1+\varepsilon_\omega)^2}\left(2\frac{1-\varepsilon_\omega}{1+\varepsilon_\omega} (n-\ell) \chi_{\parallel}^s \varepsilon_{\omega} \right) n \cos(n\theta)\\
 &-\frac{ 16\pi E^2 \ell^2 r_0^{2\ell-3}}{(1+\varepsilon_\omega)^2}\left[\frac{\varepsilon_{\omega}}{1+\varepsilon_{2\omega}} \ell \left( \chi_{\perp}^s \varepsilon_{\omega} \varepsilon_{2\omega} - \chi^s_{\parallel}\varepsilon_{2\omega} + \chi^s_{\parallel}\right) + (n+\ell-1)\chi_{\parallel}^s\varepsilon_{\omega} \right](n+2\ell)\cos((n+2\ell)\theta).
\end{align*}

Solving equation \eqref{eq:u2-omega-first-density} with $I_3$ and $I_4$, and then using \eqref{eq:u2-omega-first}, we get
\begin{align}\label{eq:u2omega-1-Inhomogeneous}
  u_{2\omega}^{(1)} =& M_{|n-2\ell|} \frac{1}{r^{|n-2\ell|}}\cos((n-2\ell)\theta) + M_n \frac{1}{r^{n}}\cos(n\theta) + M_{n+2\ell}\frac{1}{r^{n+2\ell}}\cos((n+2\ell)\theta),
\end{align}
where
\begin{align*}
   M_{|n-2\ell|} &=  \frac{8\pi E^2 \ell^2 r_0^{2\ell-2}}{(1+\varepsilon_{2\omega})(1+\varepsilon_\omega)^2}\Big(\ell \frac{ \varepsilon_{2\omega}-3}{1+\varepsilon_{2\omega}}\chi_{\perp}^s\varepsilon^2_{\omega}\varepsilon_{2\omega} + \frac{(n-\ell-1)\varepsilon_\omega -(3n -3\ell +1)}{1+\varepsilon_\omega}\chi_{\perp}^s\varepsilon^2_{\omega}\varepsilon_{2\omega} \\
  &\quad + 2\frac{(n-2\ell+1)+(n+2\ell+1)\varepsilon_{2\omega}}{1+\varepsilon_{2\omega}}\chi_{\parallel}^s \varepsilon_{\omega}  \Big)r_0^{n-2\ell}, \quad n> 2\ell, \\
  M_n &= \frac{ 16\pi E^2 \ell^2 r_0^{2\ell-2}}{(1+\varepsilon_{2\omega})(1+\varepsilon_\omega)^2}\left((\ell-1)\chi_{\perp}^s \varepsilon^2_{\omega} \varepsilon_{2\omega} - \frac{1-\varepsilon_{\omega}}{1+\varepsilon_{\omega}}(n-\ell)\left(\chi_{\perp}^s\varepsilon^2_{\omega}\varepsilon_{2\omega}+2 \chi_{\parallel}^s \varepsilon_{\omega}\right)\right)r_0^n,  \\
   M_{n+2\ell} &= \frac{8\pi E^2 \ell^2 r_0^{2\ell-2}}{(1+\varepsilon_{2\omega})(1+\varepsilon_\omega)^2}(n+2\ell-1) \left(\chi_{\perp}^s\varepsilon^2_{\omega}\varepsilon_{2\omega}+2 \chi_{\parallel}^s \varepsilon_{\omega}\right)r_0^{n+2\ell}.
\end{align*}
It follows that for $n < 2\ell$,
\begin{align*}
  M_{|n-2\ell|} &=  \frac{8\pi E^2 \ell^2 r_0^{2\ell-2}}{(1+\varepsilon_{2\omega})(1+\varepsilon_\omega)^2}\Big(\ell \chi_{\perp}^s\varepsilon^2_{\omega}\varepsilon_{2\omega} 
  +  \frac{(n-\ell-1)\varepsilon_\omega  -(3n -3\ell +1)}{1+\varepsilon_\omega}\chi_{\perp}^s\varepsilon^2_{\omega}\varepsilon_{2\omega}\\
  &\quad - 2(n-2\ell+1)\chi_{\parallel}^s \varepsilon_{\omega} \Big)r_0^{2\ell-n}.
\end{align*} 
 
As can be seen from \eqref{eq:u2omega-1-Inhomogeneous}, dipole radiation occurs when $|n-2\ell| = 1$.  The coefficients $M_{|n-2\ell|}$, $M_n$ and $M_{n+2\ell}$ indicate that the resonance behavior is the same as in Theorem \ref{thm:mainthm3}.

This completes the proof.
\end{proof}

\section{Conclusion}\label{sec:conclusion}
In this paper, we have investigated nonlinear harmonic generation in the context of plasmonics and established a comprehensive mathematical framework to elucidate the relationship between symmetry and nonlinear optical responses. Our findings demonstrate that second harmonic generation is most efficient when the overall geometry and background field exhibit broken symmetry and possesses the lowest degree of symmetry. Furthermore, surface plasmon resonance conditions can lead to a substantial enhancement of second harmonic signal intensity.  This understanding is crucial for the rational design of metal nanostructures with optimized geometries and background field to boost specific nonlinear properties. Therefore, plasmon-enhanced nonlinear harmonic generation offers promising avenues for advancing both the control and signal intensity in nonlinear optical applications. Nevertheless, it is important to note that with increasing nanostructure size, retardation effects become non-negligible. Future work should therefore account for retardation effects in second harmonic generation, an important direction that remains to be explored.

\section*{Acknowledgement}
The research of Z. Miao was supported by the Hong Kong Scholars Program grant XJ2024057. The research of H. Liu was supported by the Hong Kong RGC General Research Funds (projects 11311122,  11304224, and 11303125). The research of G. Zheng was supported by the NSF of China (12271151).

\section*{Data availability statement}

This is a piece of theoretical work and no data was involved.

\end{document}